%% file: article.tex
\documentclass[pdflatex,sn-mathphys-num]{sn-jnl}

\usepackage{graphicx}
\usepackage{amsmath,amssymb,amsfonts}
\usepackage{mathtools}
\usepackage{amsthm}
\usepackage{thmtools}
\usepackage{mathrsfs}
\usepackage[title]{appendix}
\usepackage{textcomp}
\usepackage{booktabs}
\usepackage{enumitem}
\usepackage{lmodern}
\usepackage{anyfontsize}
\usepackage{pgfplots}
\pgfplotsset{compat=newest}
\usepackage{tikz}
\usetikzlibrary{calc, intersections, decorations.markings, arrows.meta, pgfplots.fillbetween}

\definecolor{ggreen}{RGB}{79,145,19}
\definecolor{bblue}{RGB}{5,148,230}
\definecolor{vviol}{RGB}{168,13,219}
\definecolor{rred}{RGB}{219,54,7}
\definecolor{fillblue}{RGB}{57,199,220}
\definecolor{yyellow}{RGB}{199,137,58}

\declaretheoremstyle[spaceabove=6pt, spacebelow=9pt, postheadspace=1em, qed=\qedsymbol]{SpringerNat}

\theoremstyle{SpringerNat}
\newtheorem{theorem}{Theorem}[section]
\newtheorem{lemma}[theorem]{Lemma}
\newtheorem{proposition}[theorem]{Proposition}
\newtheorem{definition}[theorem]{Definition}
\newtheorem{remark}[theorem]{Remark}

\newtheorem{example}[theorem]{Example}
\numberwithin{equation}{section}

\DeclareMathOperator{\Res}{Res}
\DeclareMathOperator{\Int}{Int}
\DeclareMathOperator{\Ext}{Ext}
\DeclareMathOperator{\argg}{arg}

\newcommand{\C}{
	\ensuremath{\mathbb{C}}
}
\newcommand{\Hol}{
	\ensuremath{\mathcal{O}}
}
\newcommand{\N}{
	\ensuremath{\mathbb{N}}
}

\newcommand{\R}[1][]{
	\ensuremath{\mathbb{R}^#1}
}

\makeatletter
\newcommand{\proofpart}[2]{%
	\par
	\addvspace{\medskipamount}%
	\noindent\textbf{Step #1:} #2\par\nobreak
	\addvspace{\smallskipamount}%
	\@afterheading
}
\makeatother

\begin{document}

\title{Separatrix configurations in holomorphic flows}

\author*[1]{\fnm{Nicolas} \sur{Kainz}} \email{nicolas.kainz@uni-ulm.de}
\author[1]{\fnm{Dirk} \sur{Lebiedz}} \email{dirk.lebiedz@uni-ulm.de}

\affil[1]{\orgdiv{Institute of Numerical Mathematics}, \orgname{Ulm University}, \country{Germany}}

\abstract{

    We investigate properties of boundary orbits (separatrices) of canonical regions (basins/neighbourhoods of equilibria) in holomorphic flows with real-valued time. We establish the continuity of transit times along these boundary orbits and classify possible path components of the boundary of flow-invariant domains. Thus, we provide central tools for topological and geometric constructions aimed at examining the role of blow-up scenarios in separatrix configurations of basins of simple equilibria and global elliptic sectors:
	
	First, we prove that the separatrices of basins of centers is entirely composed of double-sided separatrices with a blow-up in finite positive \textit{and} finite negative time.
   
    Second, we show that the separatrices of node and focus basins (sinks and sources) exhibit a finite-time blow-up in the same time direction in which the orbits within the basin tend towards the equilibrium. Additionally, we propose a counterexample to the claim in Theorem 4.3 (3) in [\textit{The structure of sectors of zeros of entire flows}, K. Broughan (2003)], demonstrating that a blow-up does not \textit{necessarily} have to occur in \textit{both} time directions.
    
    Third, we describe the boundary structure of global elliptic sectors. It consists of the multiple equilibrium, one incoming and one outgoing separatrix attached to it, and at most countably many double-sided separatrices.
    
}

\keywords{complex analytic vector field, holomorphic flow, global phase space topology, separatrix, finite-time blow-up, transit time, basin, global elliptic sector.}

\pacs[MSC Classification]{Primary: 37F75. Secondary: 37F20, 34C37.}

\maketitle

\section{Introduction}

The research on holomorphic vector fields, i.e. real time holomorphic dynamical systems of the form
\begin{align}\label{eq:planarODE}
	\dot{x}=\frac{\mathrm{d}x}{\mathrm{d}t}=F(x),\quad x\in\Omega,\;t\in\mathbb{R}
\end{align}
with $F\in\Hol(\Omega)$, $\Omega\subset\C\,$\footnote{In most cases, we assume $\Omega = \C$.}, is an active and rapidly evolving area in mathematics, attracting significant attention in recent years. Especially the phase space geometry is an important subject, cf. \cite{garijo2006local, brickman1977conformal, broughan2003structure, kainz2024local, kainz2026geometry, kainz2023planar, andronov1973qualitativetheory, lebiedz2025sensitivities, schleich2018equivalent, broughan2004riemann}. By analyzing the global dynamics of complex analytic vector fields on Riemann surfaces with either real or complex time, one can obtain information on the global phase portrait of \eqref{eq:planarODE}, cf. \cite{alvarez2017dynamics1, alvarez2017dynamics2, alvarez2024geometry, alvarez2021symmetries, lebiedz2025sensitivities}. In this context, we are particularly interested in the boundaries of \textit{canonical regions} in the phase space, which have already been introduced for general smooth planar systems by Markus \cite[Section II.3]{markus1954global} and Neumann \cite[pp. 73--75]{neumann1975classification}. Roughly speaking, canonical regions are flow-invariant domains in which orbits behave similarly from a geometrical viewpoint, and whose boundaries separate orbits with qualitatively different behavior. Therefore, these boundaries form the structural backbone of the global phase portrait. The orbits on these boundaries are typically referred to as the \textit{separatrices} of \eqref{eq:planarODE}, cf. \cite[p. 129]{markus1954global}, \cite[p. 74]{neumann1975classification}, and \cite{alvarez2017dynamics1, dumortier2006qualitative, perko2001differential}. However, a mathematical definition of "separatrix" lacks consistency in the literature. In the holomorphic case, Broughan relates the term "separatrix" to the occurrence of a finite-time blow-up, cf. \cite{broughan2003structure,broughan2003holomorphic}.\footnote{A trajectory of \eqref{eq:planarODE} has a blow-up in positive (negative) time, if it tends to infinity in finite positive (negative) time.} In this paper we want to establish whether this particular definition of a separatrix is suitable and appropriate for systems of type \eqref{eq:planarODE}.

Each equilibrium possesses a maximal "region of influence", forming a canonical region in the sense of Markus and Neumann. In \cite{kainz2026geometry}, we analyzed the geometry of these specific regions, which correspond to basins of simple equilibria and global elliptic sectors, and established several topological properties. In general, each equilibrium of \eqref{eq:planarODE} can be classified into one of the following categories:
\begin{itemize}
    \item[(i)] A \textbf{center} (simple equilibrium), where all orbits in a neighborhood are closed periodic orbits enclosing the center.
    \item[(ii)] A \textbf{focus} or \textbf{node} (simple equilibrium), which is attracting or repelling (sink or source) such that all nearby orbits tend to the equilibrium either in positive (attracting) or negative (repelling) time.
    \item[(iii)] An equilibrium with order $m\in\mathbb{N}\setminus\{1\}$, possessing a \textbf{finite elliptic decomposition} of order $2m-2$. This geometric structure is defined in \cite[Definition 4.1]{kainz2024local} and illustrated in section 4.3.
\end{itemize}
We showed that the boundary orbits of these basins and global elliptic sectors are always unbounded, cf. \cite{kainz2026geometry}. In \cite{broughan2003structure}, Broughan even claims that all these boundary orbits blow up in finite time. However, his arguments contain several gaps, which we aim to close by providing complete and detailed proofs of the results in \cite{broughan2003structure}.

\newpage

The following outline of this paper illustrates how these gaps are addressed in detail.

First, we show the continuity of transit times on the boundary of flow-invariant domains, cf. Proposition \ref{prop:boundary_approximation}. This serves as the central tool for relating the time parametrization of boundary orbits to that of orbits within the basin and the global elliptic sector, respectively. Moreover, we classify the possible types of boundary components, cf. Proposition \ref{prop:countable_orbits}.

Using these Propositions, we prove that the boundaries of centers always consist of double-sided separatrices, cf. Theorem \ref{thm:separatrices_center}. Furthermore, we demonstrate that the boundary orbits of nodes and foci blow up in finite time in the same time direction in which the orbits within the basin approach the equilibrium, cf. Theorem \ref{thm:separatrices_node}. In addition, we present a counterexample to \cite[Theorem 4.3 (3)]{broughan2003structure}, showing that boundary orbits of nodes and foci need not blow up in finite time in \textit{both} time directions, cf. Example \ref{ex:conterexample_nodes}. We further establish that the boundaries of global elliptic sectors are formed by the multiple\footnote{We call an equilibrium a \textit{multiple equilibrium}, if it is not simple, i.e. if its order is at least $2$.} equilibrium, one incoming and one outgoing separatrix attached to it, and at most countably many double-sided separatrices, cf. Theorem \ref{thm:separatrices_sectors}. To support the technical arguments in these proofs, we provide several illustrative figures that highlight the geometric structure. Finally, we present two examples to illustrate our results.

We provide necessary contextual summaries and technical preliminaries throughout this paper. In order to keep track of the central achievements of this work, we highlight three Theorems as our main contribution:

\begin{enumerate}
  \item \textbf{Theorem \ref{thm:separatrices_center}.}\quad A center has a separatrix configuration consisting of at most countably many double-sided separatrices whose total transit time is bounded by the period of the center.
  \item \textbf{Theorem \ref{thm:separatrices_node}.}\quad A node or focus has a separatrix configuration consisting of countably many path components, each consisting of equilibria and separatrices whose qualitative type (positive or negative) depends on the stability of the equilibrium.
  \item \textbf{Theorem \ref{thm:separatrices_sectors}.}\quad A global elliptic sector has a separatrix configuration consisting of the multiple equilibrium, two separatrices of opposite type (one positive and the other negative), and at most countably many additional double-sided separatrices.
\end{enumerate}

\section{Basic definitions and notations}

For clarity and consistency, we recall basic definitions and notations that will be used throughout this work in the context of dynamical systems and related areas.

A domain is an open and connected set. The connected components of a topological space $\Omega$ are called the \textit{components} of $\Omega$, while the path-connected components of $\Omega$ are referred to as its \textit{path components}, cf. \cite[§25]{munkres2000topology}.

A \textit{trajectory} or \textit{orbit} through $x_0 \in \Omega \subset \C$ corresponding to \eqref{eq:planarODE} is defined as the maximal phase curve $\Gamma(x_0) := x(I)$, where $x$ denotes the unique solution of \eqref{eq:planarODE} with initial condition $x(0)=x_0$ and $I=I(x_0)\subset\R{}$ is its maximum interval of existence. We distinguish the forward and backward parts of the orbit -- the positive and negative semi-orbits -- by $\Gamma_+(x_0):=x(I\cap[0,\infty))$ and $\Gamma_-(x_0):=x(I\cap(-\infty,0])$. The orbit can be parametrized by the flow $\Phi(t,x_0):=x(t)$ for $t \in I$, leading to $\Gamma(x_0)=\Phi(I(x_0),x_0)$. The set of all equilibria of \eqref{eq:planarODE} is $F^{-1}(\{0\})$. By the Identity Theorem, if $F\in\Hol(\Omega)$ and $F\not\equiv 0$, then $F^{-1}(\{0\})$ is a discrete set and has no accumulation points.

Furthermore, if $x_0\in\Omega\subset\C{}$ is an initial value, $\Gamma=\Gamma(x_0)\subset\C{}$ the orbit of \eqref{eq:planarODE} through $x_0$, and $I=I(x_0)$ the maximum interval of existence, then we define the positive (negative) \textit{limit set} as
\begin{align*}
	\omega_{+(-)}(\Gamma):=\left\{v\in\C{}:\exists\,(t_k)_{k\in\N{}}\subset I\text{ with }t_k\overset{k\to\infty}{\longrightarrow}(-)\infty\text{ and }\Phi(t_k,x_0)\overset{k\to\infty}{\longrightarrow}v\right\}\text{.}
\end{align*}
This set does not depend on the initial value, i.e. $\omega_\pm(\Gamma(\tilde{x}_0))=\omega_\pm(\Gamma(x_0))$ for all $\tilde{x}_0\in\Gamma(x_0)$. Additionally, if $I(x_0)$ is bounded from above (below), then we have $\omega_{+(-)}(\Gamma(x_0))=\emptyset$, cf. cf. \cite[\S4]{andronov1973qualitativetheory}.

The Jordan curve Theorem, cf. \cite[Theorem 63.4]{munkres2000topology}, will be used throughout in this paper: If $\Gamma\subset\C{}$ is a closed Jordan curve, we denote the two components resulting from the Jordan curve Theorem by $\Int(\Gamma)$ (the bounded interior of $\Gamma$) and $\Ext(\Gamma)$ (the unbounded exterior of $\Gamma$). If the closed Jordan curve $\Gamma$ lies in a simply connected domain $\Omega\subset\C{}$, then $\Int(\Gamma)\subset\Omega$.

If $\Gamma\subset\C{}$ is an arbitrary curve and $a,b\in\Gamma$, then $\Gamma(a,b)$ is the curve piece of $\Gamma$ from $a$ to $b$. The nonnegative real number $\text{len}(\Gamma)$ is the length of $\Gamma$.

\section{Boundary orbits of flow-invariant domains}

Our analysis starts with two fundamental Propositions concerning the behavior of boundary orbits of flow-invariant domains. These results form the basis for the geometric constructions established in the subsequent chapter.

\subsection{Continuity of transit times} 

In order to analyze the separatrix configurations of basins and elliptic sectors, we relate the time parametrization of boundary orbits to that of orbits in the interior of basins and sectors. The aim of this subsection is to formulate this relation within a rigorous mathematical framework by establishing the continuity of transit times on the boundary of flow-invariant domains.

\begin{definition}[transit time {\cite[Definition 3.2]{broughan2003structure}}]\label{def:transit_times}
    Let $\Omega\subset\C$ be a domain and $F\in\Hol(\Omega)$. Let $\Gamma\subset\Omega\setminus F^{-1}(\{0\})$ be an arbitrary orbit of \eqref{eq:planarODE} and $a,b\in\Gamma$.
    \begin{itemize}
		\item[(i)] The transit time $\tau(\Gamma)$ of $\Gamma$ is defined as the Lebesgue measure of the maximum interval of existence of $\Gamma$, i.e.
		\begin{align*}
			\tau(\Gamma):=\lambda(I(x))
		\end{align*}
		for an arbitrary $x\in\Gamma$.
		\item[(ii)] The transit time $\tau(a,b)$ from $a$ to $b$ is defined as
		\begin{align*}
        	\tau(a,b):=\int\limits_{\mathclap{\Gamma(a,b)}}\frac{1}{F(z)}\,\mathrm{d}z\,\text{,}
		\end{align*}
		where $\Gamma(a,b)$ is the curve piece of $\Gamma$ from $a$ to $b$ parameterized via $\eqref{eq:planarODE}$.
	\end{itemize}
\end{definition}
\begin{remark}
    The transit time $\tau(\Gamma)$ does not depend on the choice of $x\in\Gamma$. If $I(x)\not=\R{}$, there holds $\tau(\Gamma)=\sup I(x)-\inf I(x)\in[-\infty,\infty]$. If $x_0\in\Gamma$, $a=\Phi(t_1,x_0)$ and $b=\Phi(t_2,x_0)$, then $\tau(a,b)=t_2-t_1$.
\end{remark}
\begin{lemma}\label{lem:transit_time_supremum}
	Let $\Omega\subset\C$ be a domain, $F\in\Hol(\Omega)$ and $\Gamma\subset\Omega$ an orbit of \eqref{eq:planarODE}. Assume that $\Gamma$ is not periodic. Then
	\begin{align*}
		\tau(\Gamma)=\sup_{\mathclap{x,y\in\Gamma}}\;\tau(x,y)\text{.}
	\end{align*}
\end{lemma}
\begin{proof}
    This statement is quite obvious. A formal proof is given in the appendix.
\end{proof}
\begin{proposition}\label{prop:boundary_approximation}
    Let $F\in\Hol(\C)$, $F\not\equiv 0$, be entire and $M\subset\C$ a flow-invariant domain w.r.t. $F$. Let $\Gamma\subset\partial M\setminus F^{-1}(\{0\})$ be an arbitrary orbit of \eqref{eq:planarODE}.\footnote{Note that $\partial M=\overline{M}\cap\overline{\C\setminus M}$ is flow-invariant, cf. \cite[Lemma 6.4]{teschl2012ordinary}.} Let $x,y\in\Gamma$ with $\tau(x,y)>0$ and $\varepsilon>0$. Then there exists $\delta\in(0,\varepsilon]$ such that the following two properties are satisfied:
    \begin{itemize}
        \item[(i)] $\begin{aligned}\mathcal{B}_\delta(x)\cap\mathcal{B}_\delta(y) =\emptyset\end{aligned}$.
        \item[(ii)] For all orbits $\Lambda\subset M$ satisfying $\mathcal{B}_\delta(x)\cap\Lambda\not=\emptyset$ and $\mathcal{B}_\delta(y)\cap\Lambda\not=\emptyset$ and for all $x^\prime\in\mathcal{B}_\delta(x)\cap\Lambda$ and $y^\prime\in\mathcal{B}_\delta(y)\cap\Lambda$ 
    \begin{align}\label{eq:boundary_approximation_1}
		|\tau(x^\prime,y^\prime)-\tau(x,y)|<\varepsilon
	\end{align}
	and
	\begin{align}\label{eq:boundary_approximation_2}
		|\Phi(t,x)-\Phi(t,x^\prime)|<\varepsilon\quad\forall\,t\in[0,\tau(x,y)]\text{.}
	\end{align}
    \end{itemize}
\end{proposition}
\begin{proof}
    Define $K:=\Gamma(x,y)\subset\partial M$ as the curve piece of the orbit $\Gamma$ from $x$ to $y$. Since $K$ is compact and the zeros of $F$ cannot lie arbitrarily close to $K$, we can choose $0<\varepsilon_0<\text{dist}(F^{-1}(\{0\}),K)$ such that
    \begin{align*}
		O:=\bigcup_{\xi\in K}\mathcal{B}_{\varepsilon_0}(\xi)
    \end{align*}
    is a small simply connected open neighbourhood of $K$.\footnote{The set $O$ looks like an "elongated tube" away from the zeros of $F$. The flow in $O$ forms a strip region, cf. \cite{neumann1975classification,markus1954global}, and is topologically equivalent to a quadrangle in $\C$ with parallel straight lines.}
    
    Furthermore, we find $\delta_1>0$ such that $\mathcal{B}_{\delta_1}(x)\cap\mathcal{B}_{\delta_1}(y)=\emptyset$. Note that $x\not=y$, since $\tau(x,y)\not=0$. By continuity of the flow, cf. \cite[Chapter 2.4, Theorem 4]{perko2001differential}, there exists $\delta_2>0$ such that $|\Phi(t,x)-\Phi(t,z)|<\min\{\varepsilon,\varepsilon_0\}$ for all $t\in[0,\tau(x,y)]$ and $z\in\mathcal{B}_{\delta_2}(x)$. With $\zeta:=\min_{z\in\overline{O}}|f(z)|>0$, the number $\delta:=\min\left\{\varepsilon,\varepsilon_0,\delta_1,\delta_2,\frac{\zeta\varepsilon}{2}\right\}>0$ is sufficiently small for our assertions.
    
    In fact, let $\Lambda\subset M$ be an arbitrary orbit such that $\mathcal{B}_{\delta}(x)\cap\Lambda\not=\emptyset$ and $\mathcal{B}_{\delta}(y)\cap\Lambda\not=\emptyset$. Let $x^\prime\in\mathcal{B}_{\delta}(x)\cap\Lambda$ and $y^\prime\in\mathcal{B}_{\delta}(y)\cap\Lambda$ be arbitrary. Then equation \eqref{eq:boundary_approximation_2} is already satisfied, since $\delta\le\delta_2$. Define $\tilde{K}:=\Lambda(x^\prime,y^\prime)\subset M$ as the curve piece of the orbit $\Lambda$ from $x^\prime$ to $y^\prime$. Let $\Xi_1\subset\mathcal{B}_{\delta}(x)$ and $\Xi_2\subset\mathcal{B}_{\delta}(y)$ be the straight connection lines from $x$ to $x^\prime$ and from $y^\prime$ to $y$, respectively. By construction and the choice of $\delta_2$, the path $\Xi:=\Xi_1\cup\tilde{K}\cup\Xi_2\cup K$ is a closed Jordan curve lying completely in $O$. Since $O$ is simply connected, $\Xi$ is null-homotopic in $O$. By applying Definition \ref{def:transit_times} and the homotopy version of Cauchy's Integral Theorem, we conclude
    \begin{align*}
		\tau(x^\prime,y^\prime)-\tau(x,y)=\int\limits_{\mathclap{\tilde{K}}}\frac{1}{F}\,\mathrm{d}z-\int\limits_{\mathclap{K}}\frac{1}{F}\,\mathrm{d}z=\underbrace{\int\limits_{\mathclap{\Xi}}\frac{1}{F}\,\mathrm{d}z}_{=0}-\int\limits_{\mathclap{\Xi_1}}\frac{1}{F}\,\mathrm{d}z-\int\limits_{\Xi_2}\frac{1}{F}\,\mathrm{d}z
    \end{align*}
	and thus
    \begin{align*}
		|\tau(x^\prime,y^\prime)-\tau(x,y)|=\Bigg|\int\limits_{\mathclap{\Xi_1}}\frac{1}{F}\,\mathrm{d}z+\int\limits_{\Xi_2}\frac{1}{F}\,\mathrm{d}z\Bigg|\le\frac{|x-x^\prime|}{\zeta}+\frac{|y-y^\prime|}{\zeta}<\frac{2\delta}{\zeta}\le\varepsilon\text{.}
    \end{align*}
    This proves equation \eqref{eq:boundary_approximation_1}.
\end{proof}

\subsection{Cardinality and possible types of path components}

In this section, we establish a result concerning the cardinality of the set of orbits lying on the boundary of flow-invariant domains. In addition, we characterize the possible structures of the path components of the boundary. The underlying idea for the subsequent proof is based on \cite[Step 7 of the proof of Theorem 3.3]{broughan2003holomorphic}.

\begin{proposition}\label{prop:countable_orbits}
	Let $F\in\Hol(\C)$, $F\not\equiv 0$, be entire and $M\subset\C$ a flow-invariant domain  w.r.t. to $F$ such that all orbits on $\partial M\setminus F^{-1}(\{0\})$ are unbounded. Then $\partial M$ has at most countably many path components, each of which is of one of the following types:
    \begin{itemize}
        \item[(i)] The path component consists of one orbit.
        \item[(ii)] The path component consists of one equilibrium.
		\item[(iii)] The path component consists of one equilibrium and one attached orbit, i.e. the orbit has the equilibrium as one of its limit sets.
		\item[(iv)] The path component consists of one equilibrium and two attached orbits, i.e. each orbit has the equilibrium as one of its limit sets.
	\end{itemize}
    Moreover, the set $\left\{\Gamma(x):x\in\partial M\right\}$ of all orbits of \eqref{eq:planarODE} on $\partial M$ is at most countable.
\end{proposition}
\begin{proof}
    Let $A$ be a path component of $\partial M$. By the Identity Theorem, $F^{-1}(\{0\})$ is a discrete set. Hence, $A$ contains either at most one equilibrium or at least one heteroclinic orbit connecting two equilibria. Since heteroclinic orbits are bounded, only the first case can occur in $A$. One impossible case remains: suppose that $A$ contains more than two unbounded orbits, all reaching an equilibrium $a \in A$ in infinite time. By the Jordan curve Theorem on $S^2$, these orbits separate $\C$ into at least three unbounded nonempty path components. Since $M$ is connected, it must lie entirely within exactly one of these path components. Consequently, at least one of the three unbounded orbits does not belong to $\partial M$, a contradiction. This shows that $A$ is indeed one of the four cases (i)--(iv).

    To prove the countability of the set $\{\Gamma(x): x \in \partial M\}$, it suffices to show that, for each of the types (i)--(iv), the set of path components of $\partial M$ belonging to that type is countable. Since $F \not\equiv 0$, there exist at most countably many equilibria, and hence at most countably many path components of types (ii)--(iv). Furthermore, by the Jordan curve Theorem on $S^2$, every path component $A$ of type (i) has the property that its unbounded orbit separates $\C$ into two disjoint, nonempty open components $A_1$ and $A_2$, i.e. $\C=A\cup A_1\cup A_2$ and $\partial A_1=\partial A_2=A$. Since $M$ is connected, it must lie entirely within either $A_1$ or $A_2$. Thus, we can define $\kappa_A\in\{A_1,A_2\}$ to be the unique component of $\C\setminus A$ with $\kappa_A\cap M=\emptyset$. By construction, we have $\partial \kappa_A = A$. This shows that, for any two disjoint path components $A, \hat{A} \subset \partial M$ of type (i), the corresponding sets $\kappa_A$ and $\kappa_{\hat{A}}$ are always disjoint, i.e., $\kappa_A \cap \kappa_{\hat{A}} = \emptyset$. Therefore, the set
    \begin{align*}
        \mathcal{A}:=\big\{\kappa_A:A\text{ is a path component of }\partial M\text{ of type (i)}\big\}
    \end{align*}
    is a family of pairwise disjoint non-empty open sets. Using the separability of $\C$, we can then apply \cite[Theorem 2.3.18]{engelking1977general} to conclude that $\mathcal{A}$ is countable. It follows that the number of path components of $\partial M$ of type (i) is also countable.
\end{proof}

\section{Separatrices as boundary orbits}

Having established some auxiliary results, we now turn to the notion of a separatrix. In general planar smooth dynamical systems, separatrices are known to form the boundaries of regions in the phase space exhibiting similar geometrical behavior. This naturally raises the question of whether, in the holomorphic setting of \eqref{eq:planarODE}, separatrices can be characterized and defined in an analytically precise manner, motivating the following definition.

\begin{definition}[Separatrix {\cite[Definition 3.1]{broughan2003structure}}]\label{def:separatrix}
	Let $F\in\Hol(\C)$ be entire, $\Gamma$ an arbitrary orbit of \eqref{eq:planarODE} and $x_0\in\Gamma$. If $I(x_0)\cap[0,\infty)\subset\R{}$ is bounded, $\Gamma$ is called a positive separatrix. If $I(x_0)\cap(-\infty,0]\subset\R{}$ is bounded, $\Gamma$ is called a negative separatrix. If $I(x_0)$ is bounded, $\Gamma$ is called a double-sided separatrix.
\end{definition}

At this point, we have more than one definition for a separatrix: the one introduced by Markus \cite{markus1954global} and Neumann \cite{neumann1975classification} as a boundary orbit of a canonical region, and the one given in Definition \ref{def:separatrix}. Throughout the remainder of this paper, the term "separatrix" will be used in the sense of Definition \ref{def:separatrix}, unless explicitly stated otherwise. In particular, references to Markus \cite{markus1954global} and Neumann \cite{neumann1975classification} indicate that separatrices are understood as boundary orbits of canonical regions.

\begin{remark}
	Whether an orbit is a positive/negative/double-sided separatrix or not, does not depend on the choice of the point $x_0$ in Definition \ref{def:separatrix}. Every separatrix is unbounded and has a blow-up.
\end{remark}

\begin{lemma}\label{lem:separatrix_transit_times}
    Let $F\in\Hol(\C)$ be entire and $\Gamma$ an arbitrary orbit of \eqref{eq:planarODE}. Then:
    \begin{itemize}
        \item[(i)] $\Gamma$ is a double-sided separatrix if and only if $\tau(\Gamma)<\infty$.
        \item[(ii)] $\Gamma$ is a positive separatrix if and only if there exists $x\in\Gamma$ such that
        \begin{align*}
            \sup_{\mathclap{y\in\Gamma_+(x)}}\;\tau(x,y)<\infty\text{.}
        \end{align*}
        \item[(iii)] $\Gamma$ is a negative separatrix if and only if there exists $x\in\Gamma$ such that
        \begin{align*}
            \inf_{\mathclap{y\in\Gamma_-(x)}}\;\tau(x,y)>-\infty\text{.}
        \end{align*}
    \end{itemize}
\end{lemma}
\begin{proof}
	The proof is straightforward and based on arguments similar to those used in the proof of Lemma~\ref{lem:transit_time_supremum}.
\end{proof}

By applying our continuity result for transit times along the boundary of flow-invariant sets, cf. Proposition \ref{prop:boundary_approximation}, we are now able to prove that the boundary of certain canonical regions in $\C$ consists of separatrices in the sense of Definition \ref{def:separatrix}. More specifically, we establish this for the basin of simple equilibria (centers, nodes, and foci) as well as for global elliptic sectors. In the following, we start with the case of a center basin.

\subsection{Separatrices on the boundary of center basins}

We introduced the center basin in our recent paper \cite{kainz2026geometry} and analyzed its geometry. For completeness, we briefly summarize these results.

\begin{definition}[{\cite[Definition 2.1]{kainz2026geometry}}]
    Let $F\in\Hol{}(\C{})$, $F\not\equiv0$, be entire and $a\in\C$ a center\footnote{cf. \cite[Definition 3.1]{kainz2024local}.} of \eqref{eq:planarODE}. The center basin $\mathcal{V}$ of $F$ in $a$ is
    \begin{align*}
        \mathcal{V}:=\{a\}\cup\left\{x\in\C:\Gamma(x)\text{ is periodic with }a\in\Int(\Gamma)\right\}.
    \end{align*}
\end{definition}

\begin{theorem}\label{thm:center_basin_properties}
    Let $F\in\Hol{}(\C{})$, $F\not\equiv0$, be entire and $a\in\C$ a center of \eqref{eq:planarODE} with its corresponding basin $\mathcal{V}$. Then:
    \begin{itemize}
        \item[(i)] $\mathcal{V}$ and $\partial\mathcal{V}$ are flow-invariant.
        \item[(ii)] $\partial\mathcal{V}\cap F^{-1}(\{0\})=\emptyset$.
        \item[(iii)] $\mathcal{V}$ is open, simply connected and unbounded.
        \item[(iv)] All orbits on $\partial\mathcal{V}$ are unbounded.
    \end{itemize}
\end{theorem}
\begin{proof}
    We established these geometrical properties in \cite[Chapter 2]{kainz2026geometry}.
\end{proof}

\newpage

\begin{proposition}\label{prop:periodic_orbits_interior_period}
    Let $F\in\Hol{}(\C)$, $F\not\equiv 0$, be entire and $\Gamma$ a periodic orbit of \eqref{eq:planarODE}. Then $\Gamma$ encloses an unique equilibrium, a center $a$, and its interior (except for the center) is entirely filled with periodic orbits, each of which also encloses $a$. Moreover, the period $T$ of $\Gamma$ is given by
    \begin{align*}
        T=\frac{2\pi\mathrm{i}}{F^\prime(a)}\text{.}
    \end{align*}
\end{proposition}
\begin{proof}
    The first statement is \cite[Corollary 5.1]{kainz2024local}. Moreover, by \cite[Corollary 4.6]{kainz2024local}, we have $F^\prime(a)\not=0$, i.e. the center $a$ in the interior of $\Gamma$ is a simple zero. The formula for the period $T$ follows from the Residue Theorem via the calculation
    \begin{align*}
          T=\int\limits_{\mathclap{\Gamma}}\frac{1}{F}\,\mathrm{d}z=2\pi\mathrm{i}\,\Res\left(\frac{1}{F},a\right)=2\pi\mathrm{i}\lim_{z\to a}\frac{z-a}{F(z)}=\frac{2\pi\mathrm{i}}{\lim\limits_{z\to a}\frac{F(z)-F(a)}{z-a}}=\frac{2\pi\mathrm{i}}{F^\prime(a)}\text{.}
    \end{align*}
    This formula can also be found in \cite[Theorem 2.3]{broughan2003holomorphic}.
\end{proof}

\begin{definition}
	Let $F\in\Hol{}(\C{})$, $F\not\equiv0$, be entire and $a\in\Omega$ a center of \eqref{eq:planarODE}. Then the period of $a$ is defined as the number
    \begin{align*}
        T(a):=\frac{2\pi\mathrm{i}}{F^\prime(a)}\text{.}
    \end{align*}
\end{definition}

We now turn to the analysis of the separatrix configuration of the center basin. The following Theorem can also be found in \cite[Theorem 4.1]{broughan2003structure}. However, the proof there contains certain gaps. In particular, it is not ensured that the $\delta_i$ are sufficiently small such that the sum of the transit times on the outermost "approximating" periodic orbit is indeed bounded by the period of the center. A proof that no overlaps occur on this "approximating" orbit is missing. For this reason, we provide a complete detailed proof here.

\begin{theorem}[Separatrix configuration of centers {\cite[Theorem 4.1]{broughan2003structure}}]\label{thm:separatrices_center}
    Let $F\in\Hol(\C)$, $F\not\equiv0$, be entire and $a\in\C$ a center of \eqref{eq:planarODE} with its corresponding basin $\mathcal{V}$. Then $\partial\mathcal{V}$ consists of at most countably many double-sided separatrices, i.e. there exists an index set $\mathcal{Q}\subset\N$ and double-sided separatrices $C_n\subset\partial\mathcal{V}$, $n\in\mathcal{Q}$, such that
    \begin{align}\label{eq:separatrix_center_union}
          \partial\mathcal{V}=\bigcup_{\mathclap{n\in\mathcal{Q}}}C_n\text{.}
    \end{align}
    Furthermore, the sum\footnote{The order of summation in \eqref{eq:center_sum_transit_times} is to be understood in the sense of \cite[Definition 8.2.1]{tao2022analysis}.} of the transit times of these separatrices is absolutely convergent and bounded by the period of $a$, i.e.
    \begin{align}\label{eq:center_sum_transit_times}
          \sum_{\mathclap{n\in\mathcal{Q}}}\tau(C_n)\le T(a)=\frac{2\pi\mathrm{i}}{F^\prime(a)}\text{.}
    \end{align}
\end{theorem}
\begin{proof}
    If $\partial\mathcal{V}=\emptyset$, nothing is to show. So we assume that $\partial\mathcal{V}$ is not empty. By Proposition \ref{prop:countable_orbits} and Theorem \ref{thm:center_basin_properties}, $\partial\mathcal{V}$ is the union of at most countably many unbounded orbits $C_n$, $n\in\mathcal{Q}\subset\N$, i.e. \eqref{eq:separatrix_center_union} holds.
	
	\proofpart{1}{Applying Proposition \ref{prop:boundary_approximation}}
	
	Let $n\in\mathcal{Q}$. We fix $\varepsilon>0$ and $x,y\in C_n$ with $\tau(x,y)>0$. By Proposition \ref{prop:boundary_approximation}, there exists $\delta\in(0,\varepsilon]$ such that for all orbits $\Lambda\subset\mathcal{V}$ satisfying $\mathcal{B}_{\delta}(x)\cap\Lambda\not=\emptyset$ and $\mathcal{B}_{\delta}(y)\cap\Lambda\not=\emptyset$ it holds that
    \begin{align*}
		|\tau(x^\prime,y^\prime)-\tau(x,y)|<\varepsilon\quad\forall\,x^\prime\in\mathcal{B}_{\delta}(x)\cap\Lambda,\;\forall\,y^\prime\in\mathcal{B}_{\delta}(y)\cap\Lambda\text{.}
    \end{align*}
	
	\proofpart{2}{All boundary orbits are double-sided separatrices}
	
	By continuity of the flow, cf. \cite[Chapter 2.4, Theorem 4]{perko2001differential}, there exists $\tilde{\delta}\in(0,\delta]$ such that $|\Phi(\tau(x,y),z)-y|<\delta$ for all $z\in\mathcal{B}_{\tilde{\delta}}(x)$. We choose $z_0\in\mathcal{B}_{\tilde{\delta}}(x)\cap\mathcal{V}$. We can apply Proposition \ref{prop:boundary_approximation} for $\Lambda:=\Gamma(z_0)\subset\mathcal{V}$, $x^\prime:=z_0\in\mathcal{B}_{\delta}(x)$ and $y^\prime:=\Phi(\tau(x,y),z_0)\in\mathcal{B}_{\delta}(y)$. Since $C_n$ is unbounded, it cannot be periodic and thus $\tau(x^\prime,y^\prime)\le T(a)$, cf. Proposition \ref{prop:periodic_orbits_interior_period}. We conclude
	\begin{align*}
		|\tau(x,y)|\le\left|\tau(x^\prime,y^\prime)\right|+\left|\tau(x,y)-\tau(x^\prime,y^\prime)\right|\le T(a)+\varepsilon\text{.}
	\end{align*}
	Since $x$ and $y$ are arbitrary, it follows by Lemma \ref{lem:transit_time_supremum}
	\begin{align*}
		\tau(C_n)=\sup_{\mathclap{x,y\in C_n}}\tau(x,y)\le\sup_{\mathclap{x,y\in C_n}}\left|\tau(x,y)\right|\le\sup_{\mathclap{x,y\in C_n}}T(a)+\varepsilon=T(a)+\varepsilon\text{.}
	\end{align*}
	Since $\varepsilon$ is arbitrary, we get $\tau(C_n)\le T(a)<\infty$, i.e. $C_n$ is a double-sided separatrix, cf. Lemma \ref{lem:separatrix_transit_times} (i). It remains to show equation \eqref{eq:center_sum_transit_times}.
	
	\proofpart{3}{Choosing $\varepsilon_0$ and $\varepsilon$ and applying Step 1}
	
	We fix $\tilde{\varepsilon}>0$, $N\in\N\setminus\{1\}$ and $\tilde{\mathcal{Q}}\subset\mathcal{Q}$ with $|\tilde{\mathcal{Q}}|=N$. Moreover, for all $n\in\tilde{\mathcal{Q}}$ we fix points $x_n,y_n\in C_n$ such that $\tau(x_n,y_n)>0$. We define the compact sets $K_n:=C_n(x_n,y_n)$, $n\in\tilde{\mathcal{Q}}$, and the number
	\begin{align*}
		\varepsilon_0:=\min_{\mathclap{\substack{i,j\in\tilde{\mathcal{Q}}\\i\not=j}}}\text{dist}(K_i,K_j)\text{.}
	\end{align*}
	By \cite[Theorem 32.2]{munkres2000topology}, we get $\varepsilon_0>0$. For all $n\in\tilde{\mathcal{Q}}$ and $\varepsilon:=\min\left\{\frac{\varepsilon_0}{4},\frac{\tilde{\varepsilon}}{N}\right\}$ we can use Step 1: There exist $\delta_n\in(0,\varepsilon]$ and an orbit $\Lambda_n\subset\mathcal{V}$ with $\mathcal{B}_{\delta_n}(x_n)\cap\Lambda_n\not=\emptyset$ and $\mathcal{B}_{\delta_n}(y)\cap\Lambda_n\not=\emptyset$ such that Proposition \ref{prop:boundary_approximation} can be applied. Since $N<\infty$, we find $n_0\in\tilde{\mathcal{Q}}$ such that $\Lambda_{n_0}$ is the outermost periodic orbit, i.e. $\Lambda_n\subset\overline{\Int(\Lambda_{n_0})}\subset\mathcal{V}$ for all $n\in\tilde{\mathcal{Q}}$. This also implies that $\mathcal{B}_{\delta_n}(x_n)\cap\Lambda_{n_0}\not=\emptyset$ and $\mathcal{B}_{\delta_n}(y_n)\cap\Lambda_{n_0}\not=\emptyset$ for all $n\in\tilde{\mathcal{Q}}$. We choose $x_n^\prime\in\mathcal{B}_{\delta_n}(x_n)\cap\Lambda_{n_0}$ and $y_n^\prime\in\mathcal{B}_{\delta_n}(y_n)\cap \Lambda_{n_0}$ and define $L_n:=\Lambda_{n_0}(x_n^\prime,y_n^\prime)$.
	
	\proofpart{4}{Impossibility of overlaps on $\Lambda_{n_0}$\footnote{This step solves the issues in the proof of \cite[Theorem 4.3]{broughan2003structure} mentioned above.}}
	
	Suppose that there exist two indices $i,j\in\tilde{\mathcal{Q}}$, $i\not=j$, such that $L_i\cap L_j\not=\emptyset$. Then we must have $\{x_i^\prime,y_i^\prime\}\cap L_j\not=\emptyset$. Assume $x_i^\prime\in L_j$. The case $y_i^\prime\in L_j$ can be led to contradiction by similar arguments. We have
	\begin{align*}
		\text{dist}(K_i,x_i^\prime)\le|x_i-x_i^\prime|\le\delta_i\le\varepsilon\le\frac{\varepsilon_0}{4}\text{.}
	\end{align*}
	Moreover, with $t:=\tau(x_j^\prime,x_i^\prime)>0$ and $\eta_1:=\Phi(t,x_j)\in K_j$, we can use \eqref{eq:boundary_approximation_2} to conclude
	\begin{align*}
		\text{dist}(K_j,x_i^\prime)\le|\eta_1-x_i^\prime|=|\Phi(t,x_j)-\Phi(t,x_j^\prime)|<\varepsilon\le\frac{\varepsilon_0}{4}\text{.}
	\end{align*}
	By the choice of $\varepsilon_0$, both inequalities lead to the contradiction
	\begin{align*}
		\varepsilon_0\le\text{dist}(K_i,K_j)\le\text{dist}(K_j,x_i^\prime)+\text{dist}(K_i,x_i^\prime)=\frac{\varepsilon_0}{4}+\frac{\varepsilon_0}{4}\le\frac{\varepsilon_0}{2}\text{.}
	\end{align*}
	Hence we indeed get $L_i\cap L_j=\emptyset$ for all $i,j\in\tilde{\mathcal{Q}}$, $i\not=j$.
	
	\proofpart{5}{Estimating the sum of all transit times on $\partial\mathcal{V}$}
	
	By Proposition \ref{prop:periodic_orbits_interior_period}, we conclude
	\begin{align*}
		\sum_{\mathclap{n\in\tilde{\mathcal{Q}}}}\tau(x_n^\prime,y_n^\prime)\le\tau(\Lambda_{n_0})\le T(a)\text{.}
	\end{align*}
	By equation \eqref{eq:boundary_approximation_1}, it follows
	\begin{align*}
		\sum_{\mathclap{n\in\tilde{\mathcal{Q}}}}\tau(x_n,y_n)\le\sum_{\mathclap{n\in\tilde{\mathcal{Q}}}}\left|\tau(x_n^\prime,y_n^\prime)\right|+\underbrace{\left|\tau(x_n,y_n)-\tau(x_n^\prime,y_n^\prime)\right|}_{<\frac{\tilde{\varepsilon}}{N}}<T(a)+\frac{N\tilde{\varepsilon}}{N}=T(a)+\tilde{\varepsilon}\text{.}
	\end{align*}
	Since $x_n$ and $y_n$, $n\in\tilde{\mathcal{Q}}$, are arbitrary, it follows by Lemma \ref{lem:transit_time_supremum}
	\begin{align*}
		\sum_{\mathclap{n\in\tilde{\mathcal{Q}}}}\tau(C_n)&=\sum_{n\in\tilde{\mathcal{Q}}}\, \sup_{\substack{x_n\in C_n\\ y_n\in C_n}}\tau(x_n,y_n)\\
		&=\sup\left\{\sum_{n\in\tilde{\mathcal{Q}}}\tau(x_n,y_n):x_n,y_n\in C_n\;\,\forall\,n\in\tilde{\mathcal{Q}}\right\}\\
		&\le T(a)+\tilde{\varepsilon}\text{.}
	\end{align*}
	In the second equality we used the fact that for all $n\in\tilde{\mathcal{Q}}$ the number $\tau(x_n,y_n)$ does not depend on the choice of $x_m,y_m\in C_m$, $m\in\tilde{\mathcal{Q}}\setminus\{n\}$. Since $\tilde{\varepsilon}$ and $N$ are arbitrary, we get equation \eqref{eq:center_sum_transit_times} for finite index sets.
	
	\proofpart{6}{Countable index sets $\mathcal{Q}$}
	
	To establish the well-definedness of the series in \eqref{eq:center_sum_transit_times} when $\mathcal{Q}$ is countable, we need to show that the series is absolutely convergent in the sense of \cite[Definition~8.2.1]{tao2022analysis}. To show this, let $\mathcal{Q}_N\subset\mathcal{Q}$ be a subset consisting of $N\in\N$ arbitrary indices in $\mathcal{Q}$. By Step 5, we get
	\begin{align*}
        \sum_{\mathclap{n\in\mathcal{Q}_N}}\tau(C_n)\leq T(a)\text{.}
	\end{align*}
	This estimate is independent of the choice of $N$ and $\mathcal{Q}_N$. Hence, it follows that
    \begin{align*}
        \sup\left\{\sum_{n\in A}\tau(C_n):A\subset\mathcal{Q},\,A\text{ is finite}\right\}\leq T(a)<\infty\text{.}
    \end{align*}
    As the transit times $\tau(C_n)$, $n\in\mathcal{Q}$, are nonnegative, we can apply \cite[Lemma 8.2.3]{tao2022analysis} to conclude the well-definedness of the series in \eqref{eq:center_sum_transit_times}. Moreover, this shows that \eqref{eq:center_sum_transit_times} holds also for countable index sets $\mathcal{Q}$.
\end{proof}

\subsection{Separatrices on the boundary of node and focus basins}

The next step is to consider the case of a node or focus (sink or source). As in the case of a center, we begin by recalling the definition of the corresponding basin together with its geometric properties.

\begin{definition}
    Let $F\in\Hol{}(\C{})$, $F\not\equiv0$, be entire and $a\in\C$ be a stable (unstable) focus or node.\footnote{cf. \cite[Definition 3.1]{kainz2024local}.} The basin of attraction (repulsion) $\mathcal{N}$ of $F$ in $a$ is
    \begin{align*}
        \mathcal{N}:=\left\{x\in\C:\omega_{+(-)}(\Gamma(x))=\{a\}\right\}\text{.}
    \end{align*}
\end{definition}

\begin{theorem}\label{thm:node_focus_basin_properties}
    Let $F\in\Hol{}(\C{})$, $F\not\equiv0$, be entire and $a\in\C$ a focus or node of \eqref{eq:planarODE} with its corresponding basin $\mathcal{N}$. Then:
    \begin{itemize}
        \item[(i)] $\mathcal{N}$ and $\partial\mathcal{N}$ are flow-invariant.
        \item[(ii)] $\partial\mathcal{N}$ consists of equilibria and unbounded orbits.
        \item[(iii)] $\mathcal{N}$ is open, simply connected and unbounded.
    \end{itemize}
\end{theorem}
\begin{proof}
    We established these geometrical properties in \cite[Chapter 4]{kainz2026geometry}.
\end{proof}

\begin{proposition}[{\cite[Proposition 4.4]{kainz2026geometry}}]\label{prop:nf_no_isolated_equilibria}
    Let $F\in\Hol{}(\C{})$, $F\not\equiv0$, be entire and $a\in\C$ a focus or node of \eqref{eq:planarODE} with its corresponding basin $\mathcal{N}$. It holds
    \begin{align*}
        \forall\,\tilde{a}\in\partial\mathcal{N}\cap F^{-1}(\{0\}):\forall\,\rho>0: \left(\mathcal{B}_\rho(\tilde{a})\cap\partial\mathcal{N}\right)\setminus\{\tilde{a}\}\not=\emptyset,
    \end{align*}
    i.e. there are no isolated points with respect to the subspace topology on $\partial\mathcal{N}$. Moreover, for all $\tilde{a}\in\partial\mathcal{N}\cap F^{-1}(\{0\})$ there exists an unbounded orbit $\Gamma\subset\partial\mathcal{N}$ with $\tilde{a}\in\omega_+(\Gamma)\cup\omega_-(\Gamma)$, i.e. all equilibria on $\partial\mathcal{N}$ are attached to an orbit on $\partial\mathcal{N}$.
\end{proposition}
\begin{proof}
    A detailed proof can be found in the appendix of \cite{kainz2026geometry}.
\end{proof}

In what follows, we state and prove the separatrix configuration of nodes and foci.

\begin{theorem}[Separatrix configuration of nodes and foci]\label{thm:separatrices_node}
	Let $F\in\Hol(\C)$, $F\not\equiv0$, be entire and $a\in\C$ a node or focus of \eqref{eq:planarODE} with its corresponding basin $\mathcal{N}$. Then the path components of $\partial\mathcal{N}$ can be indexed by an at most countable index set $\mathcal{Q}\subset\N$, i.e. the path components $\{C_n\}_{n\in\mathcal{Q}}$ satisfy
	\begin{align}\label{eq:separatrix_node_union}
		\partial\mathcal{N}=\bigcup_{\mathclap{n\in\mathcal{Q}}}C_n\text{.}
	\end{align}
	Furthermore, for all $n\in\mathcal{Q}$ the path component $C_n\subset\partial\mathcal{N}$ is of one of the following types:
	\begin{enumerate}[label=(\Alph*)]
        \item \label{itm:node_case_A} The set $C_n$ consists of one separatrix $\Gamma_n^{[1]}$. This separatrix is positive (negative) if and only if $a$ is stable (unstable).
		\item \label{itm:node_case_B} The set $C_n$ consists of one separatrix $\Gamma_n$ and one attached equilibrium $a_n$. This separatrix is positive (negative) if and only if $a$ is stable (unstable).
		\item \label{itm:node_case_C} The set $C_n$ consists of two separatrices $\Gamma_n^{[1]}$ and $\Gamma_n^{[2]}$ and one equilibrium $a_n$ attached to these separatrices. Both separatrices are positive (negative) if and only if $a$ is stable (unstable).
	\end{enumerate}
\end{theorem}
\begin{proof}
	We assume w.l.o.g. $\partial\mathcal{N}\not=\emptyset$.
    
    \proofpart{1}{Only the cases \ref{itm:node_case_A}, \ref{itm:node_case_B} and \ref{itm:node_case_C} are geometrically possible}
    
    By Proposition \ref{prop:countable_orbits} and Theorem \ref{thm:node_focus_basin_properties}, $\partial\mathcal{N}$ consists of at most countably many unbounded orbits and countably many equilibria within countably many path components. Moreover, Proposition \ref{prop:nf_no_isolated_equilibria} ensures that all path components of $\partial\mathcal{N}$ contain at least one unbounded orbit, i.e. the case of a single equilibrium in Proposition \ref{prop:countable_orbits} (ii) cannot occur. Hence, there indeed exists a at most countable index set $\mathcal{Q}\subset\N$ such that the path-components $\{C_n\}_{n\in\mathcal{Q}}$ satisfy equation \eqref{eq:separatrix_node_union}.
    
    In the following, we assume w.l.o.g. that $a$ is stable. The unstable case can be proven analogously by reversing the direction of time. Let $n\in\mathcal{Q}$ be arbitrarily fixed. We show that the orbits in the cases \ref{itm:node_case_A}, \ref{itm:node_case_B} and \ref{itm:node_case_C} are not only unbounded, but even separatrices.
    
    \proofpart{2}{Finding an appropriate upper bound}
    
    Let $\Gamma_n\subset C_n$. We fix a point $x\in\Gamma_n$ and choose $r_1,r_2>0$ small enough such that $\mathcal{B}_{r_1}(a)\subset\mathcal{N}$, $\mathcal{B}_{r_2}(x)\cap F^{-1}(\{0\})=\emptyset$ and $\mathcal{B}_{r_1}(a)\cap\mathcal{B}_{r_2}(x)=\emptyset$. Moreover, we choose a circle without contact $C\subset\mathcal{B}_{r_1}(a)$ around $a$, cf. \cite[\S3, 10.-14., \S7, 1.-2. and \S18, Lemma 3]{andronov1973qualitativetheory}, i.e. $C$ is a continuously differentiable closed path being nowhere tangential to $F$ and satisfying $\Int(C)\cap F^{-1}(\{0\})=\{a\}$.\footnote{In particular, from the equations (6) and (11) in \cite[\S7, 1.]{andronov1973qualitativetheory} and the remarks made in \cite[\S7, 2.]{andronov1973qualitativetheory} it follows that $C$ can be chosen as a linear transformed circle or ellipse.} Additionally, every orbit in $\mathcal{N}$ crosses $C$ exactly once, cf. \cite[\S3, 10., Figure 54]{andronov1973qualitativetheory}. Moreover, we choose a transversal $l\subset\mathcal{B}_{r_2}(x)$ through $x$, cf. \cite[\S 3]{andronov1973qualitativetheory}, as well as $\xi\in l\cap\mathcal{N}$ and define $\zeta\in\Gamma(\xi)$ as the intersection point of $\Gamma(\xi)$ with $C$, i.e. $\Gamma(\xi)\cap C=\{\zeta\}$. We define $L_1:=\text{len}(l)>0$ and $L_2:=\text{len}(C)>0$ as the lengths of $l$ and $C$, respectively, as well as
    \begin{align*}
		b_1:=\min\left\{|F(z)|:z\in l\right\}>0,\qquad b_2:=\min\left\{|F(z)|:z\in C\right\}>0\text{.}
    \end{align*}
    Then the number
    \begin{align*}
		M:=|\tau(\xi,\zeta)|+\frac{L_1}{b_1}+\frac{L_2}{b_2}>0
    \end{align*}
    will be an upper bound for the transit time on $\Gamma_+(x)\subset\Gamma_n$.
    
    \proofpart{3}{Applying Proposition \ref{prop:boundary_approximation}}
    
    Let $y\in\Gamma_+(x)\setminus\{x\}$ be arbitrary. We show that $\tau(x,y)\le M$. Let $\varepsilon\in(0,\text{dist}(\Gamma_n,C))$ be arbitrary. By Proposition \ref{prop:boundary_approximation}, there exists $\delta\in(0,\varepsilon]$ such that $\mathcal{B}_\delta(x)\cap\mathcal{B}_\delta(y)=\emptyset$ and for all orbits $\Lambda\subset\mathcal{N}$ satisfying $\mathcal{B}_\delta(x)\cap\Lambda\not=\emptyset$ and $\mathcal{B}_\delta(y)\cap\Lambda\not=\emptyset$ it holds that
    \begin{align*}
		|\tau(x^\prime,y^\prime)-\tau(x,y)|<\varepsilon\quad\forall\,x^\prime\in\mathcal{B}_{\delta}(x)\cap\Lambda,\;\forall\,y^\prime\in\mathcal{B}_{\delta}(y)\cap\Lambda\text{.}
    \end{align*}
	and
    \begin{align*}
		|\Phi(t,x)-\Phi(t,x^\prime)|<\varepsilon\quad\forall\,t\in[0,\tau(x,y)]\text{.}
    \end{align*}
    By continuity of the flow, cf. \cite[Chapter 2.4, Theorem 4]{perko2001differential}, there exists $\tilde{\delta}\in(0,\delta]$ such that $|\Phi(\tau(x,y),\xi^\prime)-y|<\delta$ for all $\xi^\prime\in\mathcal{B}_{\tilde{\delta}}(x)$. Since $x\in\partial\mathcal{N}$, there exists a point $\xi^\prime\in\mathcal{B}_{\tilde{\delta}}(x)\cap\mathcal{N}\cap l$. Hence, by choosing $\Lambda:=\Gamma(\xi^\prime)\subset\mathcal{N}$, $x^\prime:=\xi^\prime\in\mathcal{B}_{\delta}(x)$ and $y^\prime:=\Phi(\tau(x,y),\xi^\prime)\in\mathcal{B}_{\delta}(y)$, we can apply Proposition \ref{prop:boundary_approximation}. Let $\zeta^\prime$ be the intersection point of $\Lambda$ with $C$, i.e. $\Lambda\cap C=\{\zeta^\prime\}$. Since $\varepsilon<\text{dist}(\Gamma_n,C))$, we get $\zeta^\prime\in\Gamma_+(y^\prime)$ and
    \begin{align}\label{eq:node_focus_estimation}
        \tau(x,y)\le|\tau(\xi^\prime,y^\prime)|+|\tau(\xi^\prime,y^\prime)-\tau(x,y)|<|\tau(\xi^\prime,\zeta^\prime)|+\varepsilon\text{.}
    \end{align}
    Let $\Lambda_1$ be the piece of $l$ connecting $\xi$ with $\xi^\prime$ and $\Lambda_2$ be an the piece of $C$ connecting $\zeta$ with $\zeta^\prime$. Moreover, let $\Xi\subset\Gamma(\xi)$ and $\Xi^\prime\subset\Gamma(\xi^\prime)$ be the curve connecting $\xi$ and $\xi^\prime$ to $\zeta$ and $\zeta^\prime$, respectively. By construction, $J:=\Xi\cup\Lambda_2\cup\Xi^\prime\cup\Lambda_1$ is a closed Jordan curve lying completely in $\mathcal{N}$, cf. Figure \ref{fig:geom_visual_node_focus}.

    \proofpart{4}{$a\in\Ext(J)$}
    
    By construction, $a\not\in J$. Suppose, $a\in\Int(J)$. Since $\Gamma(\xi^\prime)\subset\mathcal{N}$, we get $\Gamma_+(\xi^\prime)\cap\Int(J)\not=\emptyset$. As $\Gamma_+(\xi^\prime)\setminus\{\xi^\prime\}$ cannot cross any orbit, it must have an intersection point with $J\setminus(\Xi\cup\Xi^\prime)=\Lambda_1\cup\Lambda_2\subset l\cup C$. But $\Gamma_-(\xi^\prime)$ has already an intersection point with $l$ and $C$. Thus, $\Gamma(\xi^\prime)$ would have two intersection points with $l$ or $C$, which is impossible. Hence, we get $a\in\Ext(J)$.
    
    \begin{figure}[ht]
		\begin{center}
            \tikzset{every picture/.style={line width=0.65pt}}
		  \begin{tikzpicture}[x=0.75pt,y=0.75pt,scale=0.45,rotate=95]
                \def\CircleWithoutContact{(502.33,355) .. controls (519.76,367.07) and (512.48,409.7) .. (495.83,437.09) .. controls (487.65,450.56) and (477.21,460.34) .. (466.33,461) .. controls (433.33,463) and (383.33,389) .. (395.9,359.13) .. controls (408.47,329.25) and (476.33,337) .. (502.33,355) -- cycle};
                \def\PathWithoutContact{(122.33,535) -- (195.33,585)};
                \def\OrbOne{(440.02,394.08) .. controls (648.33,407) and (564.33,251) .. (488.65,192.09) .. controls (383.33,103) and (146.16,547.02) .. (135.77,642.36)};
                \def\OrbTwo{(439.22,392.06) .. controls (496.33,391) and (527.33,307) .. (470.33,278) .. controls (382.33,226) and (173.33,527) .. (158.33,623)};
                \draw[gray, name path = pwc1] \PathWithoutContact node[below,xshift=6mm,yshift=-7mm] {$l$};
				\draw[vviol, name path = cwc] \CircleWithoutContact;
                \node at (450,480) {\textcolor{gray}{\textcolor{vviol}{$C$}}};
				\draw[gray, name path = pwc2] (297.96,241.58) -- (343.33,293);
				\draw[name path=orbit1, postaction={decorate}, decoration={markings, mark= at position 0.398 with {\arrowreversed{Classical TikZ Rightarrow}}, mark= at position 0.73 with {\arrowreversed{Classical TikZ Rightarrow}}}, line width=0.8, line join = round, line cap = round] \OrbOne;
                \node at (510,180) {$\Lambda$};
				\draw[name path=orbit2, postaction={decorate}, decoration={markings, mark= at position 0.3 with {\arrowreversed{Classical TikZ Rightarrow}}, mark= at position 0.78 with {\arrowreversed{Classical TikZ Rightarrow}}}, line width=0.8, line join = round, line cap = round] \OrbTwo;
				\draw [name path=separatrix, rred, postaction={decorate}, decoration={markings, mark= at position 0.1 with {\arrowreversed{Classical TikZ Rightarrow}}, mark= at position 0.45 with {\arrowreversed{Classical TikZ Rightarrow}}, mark= at position 0.85 with {\arrowreversed{Classical TikZ Rightarrow}}}, line width=0.8, line join = round, line cap = round]   (470.33,90) .. controls (370.33,111) and (62.4,628.71) .. (86.91,759.84);
                \node at (470,65) {\textcolor{rred}{$\Gamma_n$}};
                \pgfmathsetmacro{\L}{7};
                \pgfmathsetmacro{\P}{2.5};
                \fill (436.4,392.8) circle (3.4pt) node[right,xshift=0.2mm] {$a$};
				\fill[name intersections={of=orbit1 and pwc1, by={in1}}] (in1) circle (2pt);
				\pgfmathsetmacro{\AngleOne}{80};
				\draw[line width=0.1pt] ($(in1) + \L*({cos(\AngleOne-95)},{sin(\AngleOne-95)})$) -- ($(in1) + 170*({cos(\AngleOne-95)},{sin(\AngleOne-95)})$)  node[above,yshift=-0.2mm,xshift=0.3mm] {\textcolor{bblue}{$\xi^\prime$}};
				\fill[name intersections={of=orbit2 and pwc1, by={in2}}] (in2) circle (\P pt);
				\pgfmathsetmacro{\AngleTwo}{100};
				\draw[line width=0.1pt] ($(in2) + \L*({cos(\AngleTwo-95)},{sin(\AngleTwo-95)})$) -- ($(in2) + 170*({cos(\AngleTwo-95)},{sin(\AngleTwo-95)})$) node[above,yshift=-0.2mm] {\textcolor{bblue}{$\xi$}};
				\fill[name intersections={of=separatrix and pwc1, by={in3}}] (in3) circle (\P pt);
				\pgfmathsetmacro{\AngleThree}{-20};
				\draw[line width=0.1pt] ($(in3) + \L*({cos(\AngleThree-95)},{sin(\AngleThree-95)})$) -- ($(in3) + 110*({cos(\AngleThree-95)},{sin(\AngleThree-95)})$) node[right,xshift=-0.2mm] {$x$};
                \fill[name intersections={of=separatrix and pwc2, by={in4}}] (in4) circle (\P pt);
				\pgfmathsetmacro{\AngleThree}{-105};
				\draw[line width=0.1pt] ($(in4) + \L*({cos(\AngleThree-95)},{sin(\AngleThree-95)})$) -- ($(in4) + 120*({cos(\AngleThree-95)},{sin(\AngleThree-95)})$) node[below,yshift=0.2mm] {$y$};
                \fill[name intersections={of=orbit1 and pwc2, by={in5}}] (in5) circle (\P pt);
				\pgfmathsetmacro{\AngleThree}{5};
				\draw[line width=0.1pt] ($(in5) + \L*({cos(\AngleThree-95)},{sin(\AngleThree-95)})$) -- ($(in5) + 220*({cos(\AngleThree-95)},{sin(\AngleThree-95)})$) node[right,xshift=-0.2mm] {$y^\prime$};
                \fill[name intersections={of=orbit1 and cwc, by={in6}}] (in6) circle (\P pt);
				\pgfmathsetmacro{\AngleThree}{155};
				\draw[line width=0.1pt] ($(in6) + \L*({cos(\AngleThree-95)},{sin(\AngleThree-95)})$) -- ($(in6) + 160*({cos(\AngleThree-95)},{sin(\AngleThree-95)})$) node[left,xshift=0.2mm] {\textcolor{bblue}{$\zeta^\prime$}};
                \fill[name intersections={of=orbit2 and cwc, by={in7}}] (in7) circle (\P pt);
				\pgfmathsetmacro{\AngleThree}{50};
				\draw[line width=0.1pt] ($(in7) + \L*({cos(\AngleThree-95)},{sin(\AngleThree-95)})$) -- ($(in7) + 150*({cos(\AngleThree-95)},{sin(\AngleThree-95)})$) node[right,xshift=-0.2mm,yshift=0.5mm] {\textcolor{bblue}{$\zeta$}};
                \begin{scope}
                    \clip (in7) -- ($(in7) + (-20,20)$) -- (in6) -- ($(in6) + (100,0)$) -- ($(in7) + (100,0)$) -- cycle;
                    \draw[bblue,line width=0.85pt] \CircleWithoutContact;
                \end{scope}
                \begin{scope}
                    \clip (in1) -- ($(in1) + (10,-10)$) -- (in2) -- ($(in1) + (-10,20)$) -- cycle;
                    \draw[bblue,line width=0.85pt] \PathWithoutContact;
                \end{scope}
                \begin{scope}
                    \clip (in6) -- ($(in6) + (-70,-70)$) -- (in1) -- ($(in1) + (-80,-500)$) -- ($(in6) + (200,-250)$) -- ($(in6) + (200,50)$) -- cycle;
                    \draw[bblue,line width=0.85pt] \OrbOne;
                \end{scope}
                \begin{scope}
                    \clip (in7) -- ($(in7) + (-20,-30)$) -- (in2) -- ($(in2) + (-100,-400)$) -- ($(in7) + (110,-80)$) -- cycle;
                    \draw[bblue,line width=0.85pt] \OrbTwo;
                \end{scope}
                \draw[name path=orbit1, draw=none, postaction={decorate}, decoration={markings, mark= at position 0.398 with {\arrowreversed[bblue]{Classical TikZ Rightarrow}}, mark= at position 0.73 with {\arrowreversed[bblue]{Classical TikZ Rightarrow}}}, line width=0.8, line join = round, line cap = round] \OrbOne;
                \node at (510,180) {$\Lambda$};
				\draw[name path=orbit2, draw=none, postaction={decorate}, decoration={markings, mark= at position 0.3 with {\arrowreversed[bblue]{Classical TikZ Rightarrow}}, mark= at position 0.78 with {\arrowreversed[bblue]{Classical TikZ Rightarrow}}}, line width=0.8, line join = round, line cap = round] \OrbTwo;
                \fill[bblue] (in1) circle (\P pt);
                \fill[bblue] (in2) circle (\P pt);
                \fill[bblue] (in6) circle (\P pt);
                \fill[bblue] (in7) circle (\P pt);
                \fill (in3) circle (\P pt);
                \fill (in4) circle (\P pt);
                \fill (in5) circle (\P pt);
                \node at (307,405) {\textcolor{bblue}{$J$}};
		  \end{tikzpicture}
		\end{center}
		\caption{Geometrical visualization of the construction in the proof of Theorem \ref{thm:separatrices_node}, for the case where $a$ is attracting. The gray paths are transversals through $x$ and $y$ (both black), respectively. $\Gamma_n$ (red) is the separatrix. $C$ (purple) is the circle without contact with the equilibrium $a$ (black) in its interior. The interior of the closed curve $J$ (blue) is simply connected.}
		\label{fig:geom_visual_node_focus}
    \end{figure}

    \newpage
    
    \proofpart{5}{Estimating the transit time on $\Gamma_+(x)$}
    
     By Theorem \ref{thm:node_focus_basin_properties}, $\mathcal{N}$ is simply connected. By Step 3, we have $J\subset\mathcal{N}$. Moreover, by Step 4, $F$ has no zeros in $\Int(J)$. Hence, by applying the homotopy version of Cauchy's Integral Theorem, we conclude
    \begin{align*}
     \tau(\xi^\prime,\zeta^\prime)=\int\limits_{\mathclap{\Xi^\prime}}\frac{1}{F}\,\mathrm{d}z=\underbrace{\int\limits_{\mathclap{J}}\frac{1}{F}\,\mathrm{d}z}_{=0}-\int\limits_{\mathclap{\Lambda_1}}\frac{1}{F}\,\mathrm{d}z+\int\limits_{\Xi_2}\frac{1}{F}\,\mathrm{d}z-\int\limits_{\mathclap{\Lambda_2}}\frac{1}{F}\,\mathrm{d}z
    \end{align*}
	and thus
    \begin{align*}
		|\tau(\xi^\prime,\zeta^\prime)|\le \underbrace{\text{len}(\Lambda_1)}_{\le L_1}\;\underbrace{\max_{\mathclap{z\in\Lambda_1}}\frac{1}{|F(z)|}}_{\le\frac{1}{b_1}}+\left|\tau(\xi,\zeta)\right|+\underbrace{\text{len}(\Lambda_2)}_{\le L_2}\;\underbrace{\max_{\mathclap{z\in\Lambda_2}}\frac{1}{|F(z)|}}_{\le\frac{1}{b_2}}\le M\text{.}
    \end{align*}
    At this point, we realize that the estimate is valid for any $y\in\Gamma_+(x)$ and that the upper bound $M$ does not depend on the choice of $y$. Hence, by using \eqref{eq:node_focus_estimation}, it follows
    \begin{align*}
		\sup_{\mathclap{y\in\Gamma_+(x)}}\;\tau(x,y)\le\sup_{\mathclap{y\in\Gamma_+(x)}} \;|\tau(\xi^\prime,\zeta^\prime)| +\varepsilon\le M+\varepsilon\text{.}
    \end{align*}
    Since $\varepsilon$ is arbitrary, we get
    \begin{align*}
		\sup_{\mathclap{y\in\Gamma_+(x)}}\;\tau(x,y)\le M<\infty\text{.}
    \end{align*}
    By Lemma \ref{lem:separatrix_transit_times} (ii), we conclude that $\Gamma_n$ is indeed a positive separatrix.
\end{proof}

We completed the proof of Theorem \ref{thm:separatrices_node} via a step-by-step geometric construction. At this point, we note that a similar result is stated in \cite[Theorem 4.3]{broughan2003structure}. However, the proof provided there contains several substantial gaps. Our methodical approach differs in essential aspects:
\begin{enumerate}[label=(\Roman*)]
    \item \label{itm:node_gaps_1} In \cite[Theorem 4.3 (3)]{broughan2003structure}, the author additionally claims that the separatrix in case \ref{itm:node_case_A} in Theorem \ref{thm:separatrices_node} is positive \textit{and} negative, independent of the stability of the equilibrium $a$. However, no proof is provided for this assertion. In particular, it is not straightforward to adapt the final part of our proof to show the same property for $y\in\Gamma_-(x)$, since $\Lambda$ approaches the equilibrium $a$ only for $t\to\infty$. Consequently, the claim in \cite[Theorem 4.3 (3)]{broughan2003structure} turns out to be incorrect. We will propose a counterexample subsequently illustrating that the blow-up does not necessarily have to occur in \textit{both} time directions.
    \item The case of an isolated equilibrium on $\partial\mathcal{N}$ is not addressed in the 4\textsuperscript{th} step of the proof of \cite[Theorem 4.3]{broughan2003structure}. In that step, a path component $B_\lambda$ is considered, but the case $\overline{B_\lambda}=B_\lambda$ is not covered by the author. In fact, if $\overline{B_\lambda}\setminus B_\lambda=\emptyset$, i.e., if $B_\lambda$ consists of a single equilibrium, the argument in this step of the proof is not valid. This gap is closed in Proposition \ref{prop:nf_no_isolated_equilibria}, based on a detailed proof provided in the appendix of \cite{kainz2026geometry}.
    \item \label{itm:node_gaps_3} In general, the argumentation in 6\textsuperscript{th} step of the proof of \cite[Theorem 4.3]{broughan2003structure} is vague and lacks a concrete realization of the underlying idea. This issue has been resolved in our detailed proof provided above.
\end{enumerate}

\begin{example}\label{ex:conterexample_nodes}
    We now present a counterexample to the claim in \cite[Theorem 4.3 (3)]{broughan2003structure}, as described in \ref{itm:node_gaps_1}. Consider the entire vector field $F:\mathbb{C}\to\mathbb{C}$ given by
    \begin{align}\label{eq:counterexample_F}
        F(x):=x\exp(x).
    \end{align}
    By applying \cite[Theorem 3.2]{kainz2024local},  note that the only equilibrium $a=0$ of $F$ is a repelling node with basin $\mathcal{N}$.
    
    \begin{figure}[ht]
        \centering
        \includegraphics[width=0.9\textwidth]{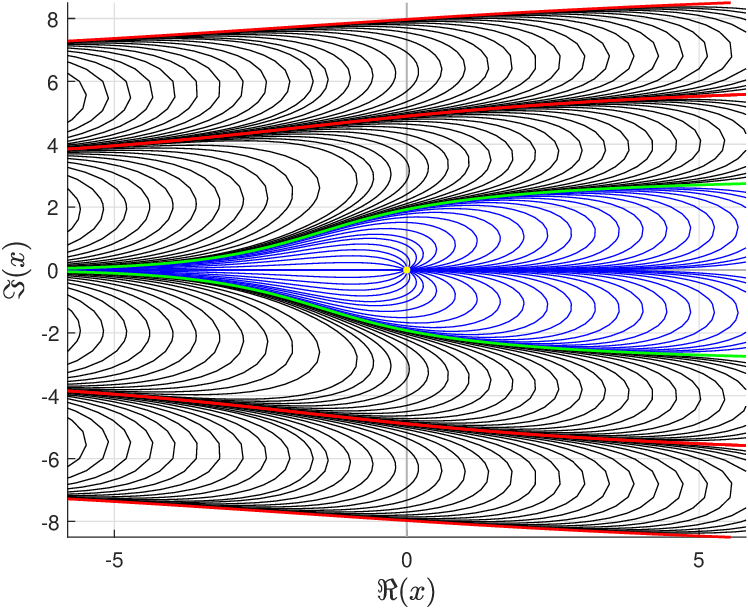}
        \caption{Local phase portrait of system \eqref{eq:planarODE} with $F(x)=xe^x$, plotted with Matlab. The equilibrium is yellow. Separatrices are red and green. The separatrices on the boundary of $\mathcal{N}$ are green. The orbits within $\mathcal{N}$ are blue. Due to the exponential term, all orbits in $\C\setminus[0,\infty)$ tend towards the left half-plane for positive time. The exponential term draws the blue and green orbits towards the negative real axis.}
        \label{fig:counterexample}
    \end{figure}
    
    By Theorem \ref{thm:separatrices_node}, the orbits on $\partial\mathcal{N}$ are negative separatrices (colored green in Figure \ref{fig:counterexample}). In particular, since $a$ is the unique equilibrium of \eqref{eq:planarODE}, case \ref{itm:node_case_A} applies: all separatrices on $\partial\mathcal{N}$ have a blow-up in negative time as they approach $+\infty$ through the right half-plane and are unbounded in positive time as they approach $-\infty$ through the left half-plane, cf. Figure \ref{fig:counterexample}. In what follows, we demonstrate that no blow-up occurs for the positive semi-orbit of $\Gamma$. By symmetry, it suffices to consider the upper green separatrix, which we denote by $\Gamma$.

    \newpage
    
    We fix the unique point $P\in\Gamma$ with the property $\Re(P)=-1$. The positive semi-orbit $\Gamma_+(P)$ is unbounded. If $x=x_1+\mathrm{i}x_2\in\overline{\mathcal{N}}$ with $x_2>0$ sufficiently small and $x_1<-1$, then a straightforward computation yields
    \begin{align}\label{eq:counterexample_F_Re_Im}
        \begin{split}
            F(x) 
             & = e^{x_1}\bigl(x_1\cos(x_2) - x_2\sin(x_2) + \mathrm{i}(x_1\sin(x_2) + x_2\cos(x_2))\bigr)\text{,}\\
            \Re(F(x)) & = e^{x_1}(x_1\cos(x_2)-x_2\sin(x_2))<x_1e^{x_1}\cos(x_2)<0\text{,}\\ 
            \Im(F(x)) & = e^{x_1}(x_1\sin(x_2)+x_2\cos(x_2))\le |x|e^{x_1}\sin(\underbrace{x_2+\arg(x)}_{\in(\pi,2\pi)})<0\text{.}
        \end{split}
    \end{align}
    
    This shows that the exponential term indeed draws the blue orbits within $\mathcal{N}$ as well as the green separatrix $\Gamma$ towards the negative real axis. Specifically, for every $y<-1$ we find a unique point $P_y\in\Gamma$ with $\Re(P_y)=y$. This observation allows the following construction: Let $\eta$ and $\eta_y$, for $y<-1$, denote the straight line segments orthogonally connecting $P$ and $P_y$, respectively, to the real axis. Let $\Gamma_y=\Gamma(P,P_y)$ be the piece of $\Gamma$ from $P$ to $P_y$. Then $J_y:=\eta\cup\Gamma_y\cup\eta_y\cup[y,-1]$ is a closed Jordan curve, cf. Figure \ref{fig:counterexample_construction}.
    
    \begin{figure}[ht]
        \begin{center}
            \begin{tikzpicture}[scale=1.1]
                 \begin{axis}[
                        axis lines=middle,
                        axis line style={-{Latex}},
                        axis on top,
                        xlabel={$\Re(x)$}, ylabel={$\Im(x)$},
                        xmin=-4.6, xmax=1.6,
                        ymin=-0.35, ymax=3.6,
                        xtick={-3,-1,1},
                        xticklabels={$y$,$-1$,$1$},
                        ytick={1,3},
                        yticklabels={$1$,$3$},
                        tick style={color=black,line width=0.85pt},
                        width=13cm, height=7cm,
                        ticklabel style={font=\small},
                        xlabel style={at={(axis cs:\pgfkeysvalueof{/pgfplots/xmax},0)},anchor=west,xshift=1pt},
                        ylabel style={at={(axis cs:0,\pgfkeysvalueof{/pgfplots/ymax})},anchor=south,yshift=1pt},
                        scale = 0.85,
                        after end axis/.code={\fill[rred] (axis cs:0,0) circle (1.3pt);},
                    ]
                    
                    \addplot[green, postaction={decorate}, decoration={markings, mark= at position 0.14 with {\arrowreversed{Classical TikZ Rightarrow}}, mark= at position 0.81 with {\arrowreversed{Classical TikZ Rightarrow}}}, line width=0.85 pt, line join = round, line cap = round] table[col sep=comma] {counterexample_construction_J.csv};
                    \addplot[bblue, line width=1 pt, line join = round, line cap = round] table[col sep=comma] {counterexample_construction_J_cut.csv};
    
                    \draw[bblue,line width=1 pt] (axis cs:-1,1.46740974227558) -- (axis cs:-1,0);
                    \draw[bblue,line width=1 pt] (axis cs:-3,0.480586597670111) -- (axis cs:-3,0);
                    \draw[bblue,line width=1 pt] (axis cs:-3,0) -- (axis cs:-1,0);
                    \fill[vviol] (axis cs:-3,0.480586597670111) circle (1.3pt);
                    \fill[vviol] (axis cs:-1,1.46740974227558) circle (1.3pt);
                    \node[green] at (axis cs:1.43,2.55) {$\Gamma$};
                    \node[bblue] at (axis cs:-0.875,0.72) {$\eta$};
                    \node[bblue] at (axis cs:-2.84,0.22) {$\eta_y$};
                    \node[bblue] at (axis cs:-2,1.23) {$\Gamma_y$};
                    \node[vviol] at (axis cs:-1,1.7) {$P$};
                    \node[vviol] at (axis cs:-3,0.75) {$P_y$};
                    \node[red] at (axis cs:0.13,0.18) {$a$};
                \end{axis}
            \end{tikzpicture}
        \end{center}
        \caption{Visualization of the construction of the closed Jordan curve $J_y$ (blue). $\Gamma_y$ (blue) connects the points $P$ and $P_y$ (purple) via the negative separatrix $\Gamma$ (green). The straight line segments $\eta$ and $\eta_y$ connect $\Gamma$ to the real axis orthogonally. The point $a$ (red) is the equilibrium. The separatrix tends towards the negative real axis.}
        \label{fig:counterexample_construction}
    \end{figure}
    
    A simple numerical computation shows that $\Im(P)<2$. Hence, by using the estimate in \eqref{eq:counterexample_F_Re_Im}, we obtain $J_y\subset[y,-1]\times[0,2]$ for all $y<-1$. This leads to 
    \begin{align*}
        \frac{1}{|F(x)|}=\frac{|e^{-x}|}{|x|}=\frac{e^{-\Re(x)}}{|x|}\le e\quad\forall\,x\in\eta
    \end{align*}
    and thus
    \begin{align}\label{eq:counterexample_eta}
        \bigg|\int\limits_{\mathclap{\eta}}\frac{1}{F}\,\mathrm{d}z\bigg|\le\text{len}(\eta)\,\max_{\mathclap{z\in\eta}}\frac{1}{|F(z)|}\le \Im(P)e<2e\text{.}
    \end{align}
    Furthermore, we calculate
    \begin{align*}
        \tau(P,P_y)=\int\limits_{\mathclap{\Gamma_y}}\frac{1}{F}\,\mathrm{d}z=\int\limits_{\mathclap{\Gamma_y}}\frac{e^{-z}}{z}\,\mathrm{d}z\text{.}
    \end{align*}
    Since $\frac{1}{F}$ is holomorphic on $\C\setminus\{0\}$ and $0\not\in\overline{\Int(J_y)}$ for all $y<-1$, the homotopy version of Cauchy’s Integral Theorem applies to $J_y$. Moreover, as $\text{len}(\eta_y)=\Im(P_y)$ tends to zero exponentially, we obtain
    \begin{align*}
        \lim_{\mathclap{y\to-\infty}}\;\tau(P,P_y)=\;\lim_{\mathclap{y\to-\infty}}\;\; \int\limits_{\mathclap{\eta}}\frac{1}{F}\,\mathrm{d}z +\int\limits_{-1}^y\frac{e^{-s}}{s}\,\mathrm{d}s\,\overset{\mathclap{\eqref{eq:counterexample_eta}}}{\ge}\,-2e+\underbrace{\int\limits_{-1}^{-\infty}\frac{e^{-s}}{s}\, \mathrm{d}s}_{\mathclap{=\int\limits_{1}^{\infty}\frac{e^{u}}{u}\,\mathrm{d}u=\infty}}=\infty\text{.}
    \end{align*}
    The last step relies on the fact that the exponential term eventually dominates the linear term in the denominator, analogous to the behavior of the exponential integral. Hence, by Lemma \ref{lem:separatrix_transit_times} (ii), $\Gamma$ cannot be a positive separatrix.\\

    We would like to point out here that there is also another approach to understanding why $\Gamma$ cannot blow up in finite positive time. For this, one has to study the complicated behavior of $F$ near the essential singularity $x=\infty$. This requires certain results deduced from the analysis of flows with complex time, cf. \cite{alvarez2017dynamics1, alvarez2024geometry, alvarez2021symmetries, heitel2019analytical}. Since this theory is very deep and a concise summary of the required results is hardly feasible within the scope of this counterexample, we chose a direct computation as examined above. Nevertheless, in what follows we briefly outline this alternative approach.
    
    The function $F$ in \eqref{eq:counterexample_F} can be seen as a complex analytic vector field on $\C$ with an essential singularity in $\infty\in\hat{\C}$, which is the Riemann sphere. In \cite[Definition 2.10 and Chapter 6.2]{alvarez2024geometry}, the authors describe the relation between real and complex time trajectories of complex analytic vector fields. A qualitative description near $a\in\hat{\C}$ and $\infty\in\hat{\C}$ is already given in \cite[Example 5.2 and Figure 6]{alvarez2024geometry}. Since $F\in\mathcal{E}(1,0,1)$, cf. \cite{alvarez2021symmetries}, we can apply \cite[Theorem 5.1]{alvarez2024geometry} to derive the existence of an hyperbolic tract over each finite asymptotic value and an elliptic tract over each infinite asymptotic value. Hence, for $\rho>0$ sufficiently small in the sense of \cite[Definition 3.2]{alvarez2024geometry}, there exist biholomorphisms $\mathcal{Y}_1:U_0(\rho)\to H$ and $\mathcal{Y}_2:U_\infty(\rho)\to E$, where $H$ and $E$ denote a hyperbolic and an elliptic sector near $\infty$, respectively, cf. \cite[Definition 4.1 (2)]{alvarez2024geometry}. In summary, this shows that the tracts with these unbounded orbits above and below the node basin in Figure \ref{fig:counterexample} are biholomophic to entire sectors, cf. \cite[Chapter 5.3.1]{alvarez2017dynamics1}. Thus, the boundary of these tracts consists of two separatrices, both tending to $\infty$ in both time directions. However, by \cite[Equation (5.3)]{alvarez2017dynamics1}, these orbits form a blow-up only in \textit{exactly one} time direction. Choosing the tract above the node basin in Figure \ref{fig:counterexample} yields a separatrix (the upper green orbit) on the boundary of the node basin, which tends to $\infty$ in both time directions but blows up only in negative time.
    
    This quantitative behavior of the orbits near infinity can also be seen from another point of view: Locally near $\infty$, there is no way to distinguish an orbit inside the node basin from one lying in the entire sector. In fact, the orbits near the set $A:=(-\infty,C]\times\{0\}\subset\C$ with $C\ll 1$ exhibit similar behavior, as parabolic sectors cannot be recognized between two local elliptic sectors. This indicates that the node does not affect the local quantitative structure (blow-up or not) of the trajectories near $A$. In other words, in this case the node cannot force the boundary orbits of the basin to blow up in finite time. Near $\infty$, the exponential term in $F$ "dominates", so that only the quantitative structure of the entire sectors remains visible.
\end{example}

\newpage

\subsection{Separatrices on the boundary of elliptic sectors}

In this section, we analyze the time behavior of orbits on the boundary of global elliptic sectors. These specific canonical regions have been introduced in \cite[Chapter 3]{kainz2026geometry}. For the sake of completeness, we give an overview of some known results about local and global elliptic sectors, which we will need later.

Local elliptic sectors have been introduced and described in detail in \cite[Chapter 4]{kainz2024local}. Roughly speaking, their geometric structure is defined by the existence of the following objects, which are illustrated in Figure \ref{fig:local_elliptic_sector}:
\phantom{.}\\
\begin{itemize}
    \item[(i)] One homoclinic orbit $\Xi$ tending to the multiple equilibrium in both time directions.
    \item[(ii)] Two orbits $\Gamma_1$ and $\Gamma_2$ attached to the multiple equilibrium.
    \item[(iii)] Two transversals $\Lambda_1$ and $\Lambda_2$ connecting $\Xi$ with $\Gamma_1$ and $\Gamma_2$, respectively.
    \item[(iv)] Two start and two end points of $\Lambda_1$ and $\Lambda_2$ denoted by $E_1,E_2\in\Xi$, $p_1\in\Gamma_1$ and $p_2\in\Gamma_2$, respectively.
\end{itemize}
\phantom{.}\\
For a local elliptic sector $S$, these objects are required to satisfy the following properties:
\phantom{.}\\
\begin{itemize}
    \item[(i)] $\partial S=\Gamma_-(p_1)\cup\Lambda_1\cup\Xi(E_1,E_2)\cup\Lambda_2\cup\Gamma_+(p_2)\cup\{a\}$.
    \item[(ii)] $\omega_+(\Gamma(x))=\omega_-(\Gamma(x))=\{a\}\;\;\forall\,x\in\Int(\Xi)$.
    \item[(iii)] $\forall\,y_1\in\Lambda_1,y_2\in\Lambda_2$:
    \begin{itemize}[label=\textbullet]
        \item $\langle F(y_1),\nu_{\Lambda_1}(y_1)\rangle>\;0$.
        \item $\langle F(y_2),\nu_{\Lambda_2}(y_2)\rangle<\;0$.
        \item $\Gamma_-(y_1)\subset S$ and $\,\omega_-(\Gamma(y_1))=\{a\}$.
        \item $\Gamma_+(y_2)\subset S$ and $\,\omega_+(\Gamma(y_2))=\{a\}$.
    \end{itemize}
\end{itemize}

\begin{figure}[ht]
    \begin{center}
        \begin{tikzpicture}
             \begin{axis}[axis lines=none, ticks=none, xmin=-0.2, xmax=2, ymin=-0.36, ymax=1.9, scale = 1]
                \pgfmathsetmacro{\L}{0.75};
                \pgfmathsetmacro{\LGray}{0.85*\L}
                \pgfmathsetmacro{\LArrows}{0.6*\L}
                \pgfmathsetmacro{\P}{1.3};
                
                \addplot[draw=none, fill=fillblue!60, opacity=0.9] table[col sep=comma] {homoclinicOrbit.csv};
                
                \addplot[color=ggreen, name path = homoclinicOrbit, postaction={decorate}, decoration={markings, mark= at position 0.25 with {\arrow{Classical TikZ Rightarrow}}, mark= at position 0.55 with {\arrow{Classical TikZ Rightarrow}}, mark= at position 0.8 with {\arrow{Classical TikZ Rightarrow}}}, line width=\L pt, line join = round, line cap = round] table[col sep=comma] {homoclinicOrbit.csv};
                \node at (0.9,0.7) [right] {\textcolor{ggreen}{$\Xi$}};
                \fill (0,0) circle (2pt) node[left,yshift=-1.4mm] {$a$};
                \path[name path=path1] (0.890285394129103,0.408775954806926) .. controls (1.25,0.1) .. (1.2,-0.2);
                \path[name path=path2] (0.530612668121931,0.747386591257611) .. controls (0.25,1.35) .. (-0.3,1.3);
                \draw[rred, name path=sep1, postaction={decorate}, decoration={markings, mark= at position 0.3 with {\arrow{Classical TikZ Rightarrow}}, mark= at position 0.85 with {\arrow{Classical TikZ Rightarrow}}}, line width=\L pt, line join = round, line cap = round] (0,0) .. controls (1.2,-0.2) .. (1.9,0.05);
                \draw[rred, name path=sep2, postaction={decorate}, decoration={markings, mark= at position 0.3 with {\arrowreversed{Classical TikZ Rightarrow}}, mark= at position 0.87 with {\arrowreversed{Classical TikZ Rightarrow}}}, line width=\L pt, line join = round, line cap = round] (0,0) .. controls (-0.15,0.9) .. (0.15,1.8);
                \fill[name intersections={of=sep1 and path1, by={in1}}] (in1) circle (\P pt) node[below,yshift=-0.3mm] {$p_1$};
                \fill[name intersections={of=sep2 and path2, by={in2}}] (in2) circle (\P pt) node[left] {$p_2$};
                
                \input{homoclinicOrbitE1Lambda1}; 
                \addplot [fillblue!60, opacity=0.9] fill between [of=top1 and sep1];
                \path[fill=white] (in1) --  (1.3,1) -- (1.9,0.05) -- (1.6,-0.25) -- cycle;
                
                \input{homoclinicOrbitE2Lambda2}; 
                \addplot [fillblue!60, opacity=0.9] fill between [of=top2 and sep2];
                
                \draw[vviol,line width=\LGray pt, line join = round, line cap = round] (0.890285394129103,0.408775954806926) .. controls (1.2,0.2) .. (in1);
                \draw[vviol,line width=\LGray pt, line join = round, line cap = round] (0.530612668121931,0.747386591257611) .. controls (0.35,1.2) .. (in2);
                \draw[-{Classical TikZ Rightarrow}, line width=\LArrows pt, line join = round, line cap = round] (0.89,0.12) .. controls (1.14,0.32) .. (1.18,0.6);
                \draw[-{Classical TikZ Rightarrow}, line width=\LArrows pt, line join = round, line cap = round] (0.98,-0.04) .. controls (1.4,0) .. (1.6,0.22);
                \draw[-{Classical TikZ Rightarrow}, line width=\LArrows pt, line join = round, line cap = round] (0.4,1.5) .. controls (0.17,1.33) .. (0.07,0.95);
                \draw[-{Classical TikZ Rightarrow}, line width=\LArrows pt, line join = round, line cap = round] (0.8,0.98) .. controls (0.51,1.02) .. (0.24,0.74);
                \fill (0.890285394129103,0.408775954806926) circle (\P pt) node[left,yshift=1mm,xshift=0.3mm] {$E_1$};
                \fill (0.530612668121931,0.747386591257611) circle (\P pt) node[below,xshift=0.8mm] {$E_2$};
                \node at (1.15,0.22) [right] {\textcolor{vviol}{$\Lambda_1$}};
                \node at (0.33,1.19) [right] {\textcolor{vviol}{$\Lambda_2$}};
                \node at (1.88,0.04) [above] {\textcolor{rred}{$\Gamma_1$}};
                \node at (0.14,1.79) [right] {\textcolor{rred}{$\Gamma_2$}};
                \node at (0.09,0.67) [below] {\textcolor{bblue}{$S$}};
                
                \draw[rred, name path=sep1, postaction={decorate}, decoration={markings, mark= at position 0.3 with {\arrow{Classical TikZ Rightarrow}}, mark= at position 0.85 with {\arrow{Classical TikZ Rightarrow}}}, line width=\L pt, line join = round, line cap = round] (0,0) .. controls (1.2,-0.2) .. (1.9,0.05);
                \draw[rred, name path=sep2, postaction={decorate}, decoration={markings, mark= at position 0.3 with {\arrowreversed{Classical TikZ Rightarrow}}, mark= at position 0.87 with {\arrowreversed{Classical TikZ Rightarrow}}}, line width=\L pt, line join = round, line cap = round] (0,0) .. controls (-0.15,0.9) .. (0.15,1.8);
                
                \fill (in1) circle (\P pt);
                \fill (in2) circle (\P pt);
                \fill (0,0) circle (2pt);
            \end{axis}
        \end{tikzpicture}
    \end{center}
    \caption{\cite[Fig. 1]{kainz2026geometry} Geometrical objects of a local elliptic sector $S$ (light blue) with counterclockwise direction in a multiple equilibrium $a$ (black). $\Gamma_1$ and $\Gamma_2$ (red) are the boundary orbits of $S$. $\Lambda_1$ and $\Lambda_2$ (purple) are the transversals. $\Xi=\Gamma(E_1)$ (green) is the homoclinic orbit. The black arrows indicate the direction of the vector field.}
    \label{fig:local_elliptic_sector}
\end{figure}

A finite elliptic decomposition (FED) can be obtained by cyclically copying the geometry in Figure \ref{fig:local_elliptic_sector} around the equilibrium. We established the existence of a FED of order $2m-2$ in a multiple equilibrium $a$ with order $m\ge 2$, cf. \cite[Proposition 4.3]{kainz2024local} and \cite[Theorem 4.4]{kainz2024local}. In particular, we showed that each local elliptic sector of this decomposition has adjacent definite directions given by
\begin{align*}
    \mathcal{E}(F,m)=\left\{\frac{\ell\pi-\arg(F^{(m)}(a))}{m-1}\mod 2\pi:\ell\in\mathbb{Z}\right\}\subset[0,2\pi)\text{.}
\end{align*}
This summary allows us to define the global elliptic sector as follows.
\begin{definition}[{\cite[Definition 3.1]{kainz2026geometry}}]\label{def:globalSector}
    \leavevmode
    \begin{itemize}
        \item[(i)] Let $\Xi\subset\C\setminus\{a\}$ be a homoclinic orbit in $a$, i.e. $\omega_+(\Xi)=\omega_-(\Xi)=\{a\}$. $\Xi$ is a \textit{sector-forming orbit} in $a$, if for all $z\in\Int(\Xi\cup\{a\})$ the orbit $\Gamma(z)$ is also homoclinic in $a$.\footnote{A construction of a parameterization for the closed Jordan curve $\Xi\cup\{a\}$ with compact time interval can be found in \cite[Remark 4.2]{kainz2024local}.}
        \item[(ii)] Let $\Xi$ be a sector-forming orbit in $a$. The \textit{global elliptic sector} $\mathcal{S}(\Xi)$ of $F$ in $a$ with respect to $\Xi$ is
        \begin{align*}
            \mathcal{S}(\Xi):=\underbrace{\Xi\cup\Int(\Xi\cup\{a\})}_{=\overline{\Int(\Xi\cup\{a\})}\setminus\{a\}}\cup\,\mathcal{S}^\prime(\Xi)
        \end{align*}
        with
        \begin{align*}
             \mathcal{S}^\prime(\Xi):=\big\{x\in\C:\,&\Gamma(x)\text{ is homoclinic in }a,\\
             &\Xi\subset\Int(\Gamma(x)\cup\{a\})\big\}\text{.}
        \end{align*}
    \end{itemize}
\end{definition}

\begin{theorem}\label{thm:sector_basin_properties}
    Let $F\in\Hol{}(\C{})$, $F\not\equiv0$, be entire and $a\in\C$ an equilibrium of \eqref{eq:planarODE} with order $m\in\N\setminus\{1\}$. Let $\mathcal{S}:=\mathcal{S}(\Xi)$ be a global elliptic sector generated by the sector-forming orbit $\Xi$. Then:
    \begin{itemize}
        \item[(i)] $\mathcal{S}$ and $\partial\mathcal{S}$ are flow-invariant.
        \item[(ii)] $\partial\mathcal{S}\cap F^{-1}(\{0\})=\{a\}$.
        \item[(iii)] $\mathcal{S}$ is open, simply connected and unbounded.
        \item[(iv)] All orbits on $\partial\mathcal{S}\setminus\{a\}$ are unbounded.
        \item[(v)] All orbits in $\mathcal{S}$ are nested and
            \begin{align}\label{eq:sector_union}
                \mathcal{S}=\bigcup_{\mathclap{x\in\mathcal{S}}}\Int(\Gamma(x)\cup\{a\})=\bigcup_{\mathclap{x\in\mathcal{S}}} \overline{\Int(\Gamma(x)\cup\{a\})}\setminus\{a\}\text{.}
            \end{align}
        \item[(vi)] $\mathcal{S}$ does not depend on the particular choice of a sector-forming orbit, that is, for $x,y\in\mathcal{S}$, we have $\mathcal{S}(\Gamma(x))=\mathcal{S}(\Gamma(y))=\mathcal{S}$.
    \end{itemize}
\end{theorem}
\begin{proof}
    We established these geometrical properties in \cite[Chapter 3]{kainz2026geometry}.
\end{proof}

\newpage

\begin{theorem}[{\cite[Corollary 3.11]{kainz2026geometry}}]\label{thm:sector_number}
    Let $F\in\Hol{}(\C{})$, $F\not\equiv0$, be entire and $a\in\C$ an equilibrium of \eqref{eq:planarODE} with order $m\in\N\setminus\{1\}$.
    \begin{itemize}
        \item[(i)] All homoclinic orbits in $a$ are sector-forming orbits.
        \item[(ii)] There exist exactly $2m-2$ distinct global elliptic sectors in $a$, each located between two adjacent definite directions given by $\mathcal{E}(F,m)$.
    \end{itemize}
\end{theorem}
\begin{proof}
    This is \cite[Corollary 3.11]{kainz2026geometry}.
\end{proof}

The following Theorem corresponds to \cite[Theorem 4.2]{broughan2003structure}. In that work, however, the author does not provide a definition of the elliptic sector, whose geometry he analyzes. Furthermore, similar to the issue described in \ref{itm:node_gaps_3}, Step 8 of the proof of \cite[Theorem 4.2]{broughan2003structure} offers only a sketch rather than a fully developed argumentation. In what follows, we present a detailed proof based on the geometric structure of the elliptic sector summarized above.

\begin{theorem}[Separatrix configuration of global elliptic sectors {\cite[Theorem 4.2]{broughan2003structure}}]\label{thm:separatrices_sectors}
	Let $F\in\Hol{}(\C{})$, $F\not\equiv0$, be entire and $a\in\C$ an equilibrium of \eqref{eq:planarODE} with order $m\in\N\setminus\{1\}$. Let $\mathcal{S}:=\mathcal{S}(\Xi)$ be a global elliptic sector generated by the sector-forming orbit $\Xi$. Then $\partial\mathcal{S}$ consists of $a$, two separatrices $\Gamma_1,\Gamma_2$ satisfying $\omega_-(\Gamma_1)=\omega_+(\Gamma_2)=\{a\}$ and at most countably many separatrices, i.e. there exists an index set $\mathcal{Q}\subset\N$ and separatrices $C_n\subset\partial\mathcal{S}$, $n\in\mathcal{Q}$, such that
    \begin{align}\label{eq:separatrix_sector_union}
		\partial\mathcal{S}=\{a\}\cup\Gamma_1\cup\Gamma_2\cup\bigcup_{\mathclap{n\in\mathcal{Q}}}C_n\text{.}
    \end{align}
	In particular, $\Gamma_1$ is a positive and $\Gamma_2$ is a negative separatrix. Moreover, for all $n\in\mathcal{Q}$ the orbit $C_n$ is a double-sided separatrix.
\end{theorem}
\begin{proof}
    By Proposition \ref{prop:countable_orbits} and Theorem \ref{thm:sector_basin_properties}, $\partial\mathcal{S}$ consists of $a$ and countably many unbounded orbits. We can indeed find a countable index set $\mathcal{Q}$ such that \eqref{eq:separatrix_sector_union} holds. In particular, the geometry near $a$ described above ensures that $\Gamma_1$ and $\Gamma_2$ exist and are unique. It remains to show that $\Gamma_1$ is a positive, $\Gamma_2$ a negative and for every $n\in\mathcal{Q}$ the orbit $C_n$ a double-sided separatrix.
    
    \proofpart{1}{Applying the geometry of the FED in $a$}
    
    Using the geometry of a local elliptic sector in $\mathcal{S}$ with boundary orbits $\Gamma_1$ and $\Gamma_2$, there exists a homoclinic sector-forming orbit $\hat{\Xi}\subset\mathcal{S}=\mathcal{S}(\Xi)$ as well as two continuously differentiable curves $\Lambda_1,\Lambda_2\subset\overline{\mathcal{S}}$ connecting $\hat{\Xi}$ to $\Gamma_1$ and $\Gamma_2$, respectively. The two curves are nowhere tangential to $F$. We denote the start and end points of $\Lambda_1$ and $\Lambda_2$ by $E_1,E_2\in\hat{\Xi}$, $p_1\in\Gamma_1$ and $p_2\in\Gamma_2$, respectively. All orbits in $\mathcal{S}\cap\Ext(\hat{\Xi}\cup\{a\})$ cross $\Lambda_1$ as well as $\Lambda_2$ exactly once. We assume w.l.o.g. that $\mathcal{S}$ has counterclockwise direction, i.e. the situation in Figure \ref{fig:local_elliptic_sector} occurs. Let $\kappa:=\Gamma_-(p_1)\cup\Lambda_1\cup\hat{\Xi}(E_1,E_2)\cup\Lambda_2\cup\Gamma_+(p_2)\cup\{a\}$ be the closed piecewise continuously differentiable Jordan curve defining the boundary of the local elliptic sector in $\mathcal{S}$, cf. Figure \ref{fig:local_elliptic_sector}.

    \proofpart{2}{Finding an appropriate upper bound}
    
    We fix a point $x\in\partial\mathcal{S}\cap\Ext(\kappa)$ and choose $r>0$ small enough such that $\mathcal{B}_r(x)\subset\Ext(\kappa)$ and $\mathcal{B}_r(x)\cap F^{-1}(\{0\})=\emptyset$. In particular, either $x\in C_n$ with $n\in\mathcal{Q}$, or $x\in\Gamma_+(p_1)\setminus\{p_1\}$, if $x\in\Gamma_1$, or $x\in\Gamma_-(p_2)\setminus\{p_2\}$, if $x\in\Gamma_2$. For $j\in\{1,2\}$, we define $L_j:=\text{len}(\Lambda_j)>0$ as the length of $\Lambda_j$ as well as
    \begin{align*}
		b_j:=\min\left\{|F(z)|:z\in \Lambda_j\right\}>0\text{.}
    \end{align*}
    Then the number
    \begin{align*}
		M:=|\tau(E_1,E_2)|+\frac{L_1}{b_1}+\frac{L_2}{b_2}>0\text{.}
    \end{align*}
    will be an appropriate upper bound for the transit time on $\Gamma(x)$.
    
    \proofpart{3}{Applying Proposition \ref{prop:boundary_approximation} for the case $x\in\Gamma_+(p_1)\subset\partial\mathcal{S}\cap\Ext(\kappa)$}
    
	Let $y\in\Gamma_+(x)$ and $\varepsilon\in(0,\frac{r}{2})$ be arbitrary. By Proposition \ref{prop:boundary_approximation}, there exists $\delta\in(0,\varepsilon]$ such that $\mathcal{B}_\delta(x)\cap\mathcal{B}_\delta(y)=\emptyset$ and for all orbits $\Lambda\subset\mathcal{S}$ satisfying $\mathcal{B}_\delta(x)\cap\Lambda\not=\emptyset$ and $\mathcal{B}_\delta(y)\cap\Lambda\not=\emptyset$ it holds that
	\begin{align*}
		|\tau(x^\prime,y^\prime)-\tau(x,y)|<\varepsilon\quad\forall\,x^\prime\in\mathcal{B}_{\delta}(x)\cap\Lambda,\;\forall\,y^\prime\in\mathcal{B}_{\delta}(y)\cap\Lambda\text{.}
	\end{align*}
	By continuity of the flow, cf. \cite[Chapter 2.4, Theorem 4]{perko2001differential}, there exists $\tilde{\delta}\in(0,\delta]$ such that $|\Phi(\tau(x,y),z_0)-y|<\delta$ for all $z_0\in\mathcal{B}_{\tilde{\delta}}(x)$. Since $x\in\partial\mathcal{S}$, there exists a point $z_0\in\mathcal{B}_{\tilde{\delta}}(x)\cap\mathcal{S}$, i.e. with $\Lambda:=\Gamma(z_0)\subset\mathcal{S}$, $x^\prime:=z\in\mathcal{B}_{\delta}(x)$ and $y^\prime:=\Phi(\tau(x,y),z_0)\in\mathcal{B}_{\delta}(y)$ we can apply Proposition \ref{prop:boundary_approximation}. By applying our results in Step 1, there exist $\xi_1,\xi_2\in\mathcal{S}$, which are the intersection points of $\Lambda$ with $\Lambda_1$ and $\Lambda_2$, respectively. We have $\xi_1\in\Gamma_-(x^\prime)$ and $\xi_2\in\Gamma_+(y^\prime)$. For $j\in\{1,2\}$, let $\Psi_j\subset\Lambda_j$ be the curve connecting $\xi_j$ to $E_j$. By construction, $J:=\hat{\Xi}(E_1,E_2)\cup\Psi_1\cup\Lambda(\xi_1,\xi_2)\cup\Psi_2$ is a closed Jordan curve lying completely in $\Lambda\cup\Int(\Lambda\cup\{a\})$. By using equation \eqref{eq:sector_union}, we conclude $\overline{\Int(J)}\subset\mathcal{S}$, cf. Figure \ref{fig:geom_visual_elliptic_sector_approximation}.
    
    \begin{figure}[ht]
		\begin{center}
		      \begin{tikzpicture}
                \begin{axis}[axis lines=none, ticks=none, xmin=-0.3, xmax=3.5, ymin=-0.67, ymax=1.9, width=10cm]
                    \pgfmathsetmacro{\L}{0.8};
                    \pgfmathsetmacro{\LGray}{0.85*\L}
                    \pgfmathsetmacro{\Llabel}{0.35*\L}
                    \pgfmathsetmacro{\P}{1.2};
                    \addplot[draw=none, fill=fillblue!60, opacity=0.9] table[col sep=comma] {homoclinicOrbit.csv};
                    \addplot[ggreen, postaction={decorate}, decoration={markings, mark= at position 0.18 with {\arrow{Classical TikZ Rightarrow}}, mark= at position 0.8 with {\arrow{Classical TikZ Rightarrow}}, mark= at position 0.8 with {\arrow{Classical TikZ Rightarrow}}}, line width=\L pt, line join = round, line cap = round] table[col sep=comma] {homoclinicOrbit.csv};
                    \fill (0,0) circle (2pt) node[left,yshift=-1.4mm] {$a$};
                    \path[name path=path1] (0.890285394129103,0.408775954806926) .. controls (1.25,0.1) .. (1.2,-0.3);
                    \path[name path=path2] (0.530612668121931,0.747386591257611) .. controls (0.25,1.35) .. (-0.3,1.3);
                    \draw[rred, name path=sep1, postaction={decorate}, decoration={markings, mark= at position 0.3 with {\arrow{Classical TikZ Rightarrow}}, mark= at position 0.87 with {\arrow{Classical TikZ Rightarrow}}}, line width=\L pt, line join = round, line cap = round] (0,0) .. controls (1,-0.3) .. (3.3,-0.06);
                    \draw[rred, name path=sep2, postaction={decorate}, decoration={markings, mark= at position 0.3 with {\arrowreversed{Classical TikZ Rightarrow}}, mark= at position 0.87 with {\arrowreversed{Classical TikZ Rightarrow}}}, line width=\L pt, line join = round, line cap = round] (0,0) .. controls (-0.15,0.9) .. (0.15,1.8);
                    \def\ApproxiOrbit{(0,0) .. controls (0.14,0) and (0.67,-0.12) .. (1.03,-0.12) .. controls (1.39,-0.12) and (2.17,-0.07) .. (2.6,0.01) .. controls (3.03,0.08) and (2.64,0.69) .. (2.41,0.91) .. controls (2.19,1.13) and (1.9,1.3) .. (1.52,1.35) .. controls (1.13,1.41) and (0.83,1.41) .. (0.42,1.09) .. controls (0,0.77) and (-0.03,0.18) .. (0,0)};
                    \draw[smooth, name path=approxOrb, postaction={decorate}, decoration={markings, mark= at position 0.08 with {\arrow{Classical TikZ Rightarrow}}, mark= at position 0.445 with {\arrow[color=yyellow]{Classical TikZ Rightarrow}}, mark= at position 0.695 with {\arrow[color=yyellow]{Classical TikZ Rightarrow}}, mark= at position 0.91 with {\arrow{Classical TikZ Rightarrow}}}, line width=\L pt, line join = round, line cap = round] \ApproxiOrbit;
                    \fill[name intersections={of=sep1 and path1, by={in1}}] (in1) circle (\P pt);
                    \fill[name intersections={of=sep2 and path2, by={in2}}] (in2) circle (\P pt);
                    
                    \input{homoclinicOrbitE1Lambda1}; 
                    \addplot [fillblue!60, opacity=0.9] fill between [of=top1 and sep1];
                    \path[fill=white] (in1) -- (1.3,0.5) -- (3.38,0.05) -- (3.38,-0.25) -- cycle;
                    
                    \input{homoclinicOrbitE2Lambda2}; 
                    \addplot [fillblue!60, opacity=0.9] fill between [of=top2 and sep2];
                    
                    \def\PathOne{(0.890285394129103,0.408775954806926) .. controls (1.2,0.2) .. (in1)};
                    \draw[vviol, name path = path11, line width=\LGray pt, line join = round, line cap = round] \PathOne;
                    \def\PathTwo{(0.530612668121931,0.747386591257611) .. controls (0.35,1.2) .. (in2)};
                    \draw[vviol, name path = path22, line width=\LGray pt, line join = round, line cap = round] \PathTwo;
                    \draw[gray, name path = pwcleft, line width=\LGray pt, line join = round, line cap = round] (1.55,-0.34) -- (1.65,0.07);
                    \draw[gray, name path = pwcright, line width=\LGray pt, line join = round, line cap = round] (2.6,-0.25) -- (2.38,0.12);
                    \fill[name intersections={of=approxOrb and path11, by={in3}}] (in3) circle (\P pt);
                    \fill[name intersections={of=approxOrb and path22, by={in4}}] (in4) circle (\P pt);
                    \fill[name intersections={of=sep1 and pwcleft, by={in5}}] (in5) circle (\P pt);
                    \fill[name intersections={of=approxOrb and pwcleft, by={in6}}] (in6) circle (\P pt);
                    \fill[name intersections={of=sep1 and pwcright, by={in7}}] (in7) circle (\P pt);
                    \fill[name intersections={of=approxOrb and pwcright, by={in8}}] (in8) circle (\P pt);
                    \fill[yyellow] (0.890285394129103,0.408775954806926) circle (\P pt);
                    \fill[yyellow] (0.530612668121931,0.747386591257611) circle (\P pt);
                    \begin{scope}
                        \clip (0.890285394129103,0.408775954806926) -- (0.530612668121931,0.747386591257611) -- (0.7,1) -- (1.5,0.5) -- cycle;
                        \addplot[yyellow, line width=\L pt, line join = round, line cap = round] table[col sep=comma] {homoclinicOrbit.csv};
                    \end{scope}
                    \begin{scope}
                        \clip (0.890285394129103,0.408775954806926) -- (0.8,0.1) -- (in3) -- (1.5,0.3) -- cycle;
                        \draw[yyellow,line width=\LGray pt, line join = round, line cap = round] \PathOne;
                    \end{scope}
                    \begin{scope}
                        \clip (0.530612668121931,0.747386591257611) -- (0.35,0.8) -- (in4) -- (0.8,1) -- cycle;
                        \draw[yyellow,line width=\LGray pt, line join = round, line cap = round] \PathTwo;
                    \end{scope}
                    \begin{scope}
                        \clip (in3) -- (0.3,0.2) -- (in4) -- (0.65,1.7) -- (3.2,1.6) -- (3,-0.2) -- (1.6,-0.35) -- cycle;
                        \draw[yyellow,line width=\L pt, line join = round, line cap = round] \ApproxiOrbit;
                    \end{scope}
                    \pgfmathsetmacro{\AngleOne}{243};
                    \draw[line width=\Llabel pt] ($(in1) + 0.12*({cos(\AngleOne)},{sin(\AngleOne)})$) -- ($(in1) + 1.5*({cos(\AngleOne)},{sin(\AngleOne)})$) node[left,xshift=1mm,yshift=-0.5mm] {$p_1$};
                    \pgfmathsetmacro{\AngleTwo}{12};
                    \draw[line width=\Llabel pt] ($(in2) + 0.022*({cos(\AngleTwo)},{sin(\AngleTwo)})$) -- ($(in2) + 0.37*({cos(\AngleTwo)},{sin(\AngleTwo)})$) node[right,xshift=-0.5mm,yshift=0.5mm] {$p_2$};
                    \pgfmathsetmacro{\AngleThree}{30};
                    \draw[line width=\Llabel pt] ($(in3) + 0.019*({cos(\AngleThree)},{sin(\AngleThree)})$) -- ($(in3) + 0.4*({cos(\AngleThree)},{sin(\AngleThree)})$) node[right,xshift=-0.6mm,yshift=0.9mm] {\textcolor{yyellow}{$\xi_1$}};
                    \pgfmathsetmacro{\AngleFour}{300};
                    \draw[line width=\Llabel pt] ($(in4) + 0.045*({cos(\AngleFour)},{sin(\AngleFour)})$) -- ($(in4) + 1.05*({cos(\AngleFour)},{sin(\AngleFour)})$) node[right,xshift=-0.5mm,yshift=0.1mm] {\textcolor{yyellow}{$\xi_2$}};
                    \pgfmathsetmacro{\AngleFive}{280};
                    \draw[line width=\Llabel pt] ($(in5) + 0.06*({cos(\AngleFive)},{sin(\AngleFive)})$) -- ($(in5) + 0.7*({cos(\AngleFive)},{sin(\AngleFive)})$) node[right,xshift=-0.5mm] {$x$};
                    \pgfmathsetmacro{\AngleSix}{15};
                    \draw[line width=\Llabel pt] ($(in6) + 0.021*({cos(\AngleSix)},{sin(\AngleSix)})$) -- ($(in6) + 0.24*({cos(\AngleSix)},{sin(\AngleSix)})$) node[right,xshift=-0.5mm,yshift=0.9mm] {$x^\prime$};
                    \pgfmathsetmacro{\AngleSeven}{280};
                    \draw[line width=\Llabel pt] ($(in7) + 0.07*({cos(\AngleSeven)},{sin(\AngleSeven)})$) -- ($(in7) + 1*({cos(\AngleSeven)},{sin(\AngleSeven)})$) node[right,xshift=-0.5mm] {$y$};
                    \pgfmathsetmacro{\AngleEight}{85};
                    \draw[line width=\Llabel pt] ($(in8) + 0.016*({cos(\AngleEight)},{sin(\AngleEight)})$) -- ($(in8) + 0.15*({cos(\AngleEight)},{sin(\AngleEight)})$) node[above,xshift=0.4mm,yshift=-0.3mm] {$y^\prime$};
                    \pgfmathsetmacro{\AngleEOne}{-7};
                    \draw[line width=\Llabel pt] ($(0.890285394129103,0.408775954806926) + 0.021*({cos(\AngleEOne)},{sin(\AngleEOne)})$) -- ($(0.890285394129103,0.408775954806926) + 0.42*({cos(\AngleEOne)},{sin(\AngleEOne)})$) node[right,xshift=-0.7mm,yshift=0.3mm] {\textcolor{yyellow}{$E_1$}};
                    \pgfmathsetmacro{\AngleETwo}{253};
                    \draw[line width=\Llabel pt] ($(0.530612668121931,0.747386591257611) + 0.1*({cos(\AngleETwo)},{sin(\AngleETwo)})$) -- ($(0.530612668121931,0.747386591257611) + 1.1*({cos(\AngleETwo)},{sin(\AngleETwo)})$) node[below,xshift=0.6mm,yshift=0.5mm] {\textcolor{yyellow}{$E_2$}};

                    \draw[rred, name path=sep1, postaction={decorate}, decoration={markings, mark= at position 0.87 with {\arrow{Classical TikZ Rightarrow}}}, line width=\L pt, line join = round, line cap = round] (0,0) .. controls (1,-0.3) .. (3.3,-0.06);
                    \draw[rred, name path=sep2, line width=\L pt, line join = round, line cap = round] (0,0) .. controls (-0.15,0.9) .. (0.15,1.8);
                    \fill (in1) circle (\P pt);
                    \fill (in2) circle (\P pt);
                    \fill[yyellow] (in3) circle (\P pt);
                    \fill[yyellow] (in4) circle (\P pt);
                    \fill (in5) circle (\P pt);
                    \fill (in6) circle (\P pt);
                    \fill (in7) circle (\P pt);
                    \fill (in8) circle (\P pt);
                    \fill[yyellow] (0.890285394129103,0.408775954806926) circle (\P pt);
                    \fill[yyellow] (0.530612668121931,0.747386591257611) circle (\P pt);
                    \node at (2.5,0.85) [right] {\textcolor{yyellow}{$J$}};
                    \node at (0.05,0.98) [right] {$\Lambda$};
                    \node at (1.11,0.25) [right] {\textcolor{vviol}{$\Lambda_1$}};
                    \node at (0.14,1.31) [right] {\textcolor{vviol}{$\Lambda_2$}};
                    \node at (3.34,-0.07) [above] {\textcolor{rred}{$\Gamma_1$}};
                    \node at (-0.146,1.82) [right] {\textcolor{rred}{$\Gamma_2$}};
                    \node at (0.8,0.07) [above] {\textcolor{ggreen}{$\hat{\Xi}$}};
                    \fill (0,0) circle (2pt);
                \end{axis}
		      \end{tikzpicture}
		\end{center}
		\caption{Geometrical visualization of the construction in Step 3 of the proof of Theorem \ref{thm:separatrices_sectors}, for the case where the local elliptic sector (light blue) in the multiple equilibrium $a$ (black) has counterclockwise direction. The gray paths are transversals through $x$ and $y$ (both black), respectively. $\Gamma_1$ and $\Gamma_2$ (red) are the separatrices. $\Lambda_1$ and $\Lambda_2$ (purple) are the transversals of the local elliptic sector. $\Xi=\Gamma(E_1)$ (green) is the homoclinic sector-forming orbit. The interior of the closed curve $J$ (yellow) is simply connected.}
		\label{fig:geom_visual_elliptic_sector_approximation}
    \end{figure}

    \proofpart{4}{Estimating the transit time on $\Gamma_+(x)$ for the case $x\in\Gamma_+(p_1)\subset\partial\mathcal{S}\cap\Ext(\kappa)$}
    
    By using the results in Step 3, $F$ has no zeros in $\Int(J)$. Hence, we get
    \begin{align*}
        \tau(\xi_1,\xi_2)=\int\limits_{\mathclap{\Lambda(\xi_1,\xi_2)}}\frac{1}{F}\,\mathrm{d}z=\underbrace{\int\limits_{\mathclap{J}}\frac{1}{F}\,\mathrm{d}z}_{=0}-\int\limits_{\mathclap{\Psi_1}}\frac{1}{F}\,\mathrm{d}z-\underbrace{\int\limits_{\mathclap{\hat{\Xi}(E_1,E_2)}}\frac{1}{F}\,\mathrm{d}z}_{=\tau(E_1,E_2)}+\int\limits_{\mathclap{\Psi_2}}\frac{1}{F}\,\mathrm{d}z\text{.}
    \end{align*}
    and thus with the homotopy version of Cauchy's Integral Theorem
    \begin{align*}
        |\tau(\xi_1,\xi_2)|\le\underbrace{\text{len}(\Psi_1)}_{\le L_1}\;\underbrace{\max_{\mathclap{z\in\Psi_1}}\frac{1}{|F(z)|}}_{\le\frac{1}{b_1}}+\left|\tau(E_1,E_2)\right|+\underbrace{\text{len}(\Psi_2)}_{\le L_2}\;\underbrace{\max_{\mathclap{z\in\Psi_2}}\frac{1}{|F(z)|}}_{\le\frac{1}{b_2}}\le M\text{.}
    \end{align*}
    As in \eqref{eq:node_focus_estimation}, it follows
    \begin{align*}
        |\tau(x,y)|&\le\tau(x^\prime,y^\prime)+ \varepsilon\le|\tau(\xi_1,\xi_2)|+\varepsilon\le M+\varepsilon\text{.}
    \end{align*}
    As in the proof of Theorem \ref{thm:separatrices_node}, we realize that the upper bound $M$ does not depend on the choice of $y$. Hence, since $\varepsilon$ is arbitrary, we conclude
    \begin{align*}
		\sup_{\mathclap{y\in\Gamma_+(x)}}\;\tau(x,y)\le M<\infty\text{.}
    \end{align*}
    By Lemma \ref{lem:separatrix_transit_times} (ii), we conclude that $\Gamma_1$ is indeed a positive separatrix.
    
    \proofpart{5}{The case $x\in\Gamma_-(p_2)\subset\partial\mathcal{S}\cap\Ext(\kappa)$}
    
    This case can be treated analogous to the case $x\in\Gamma_+(p_1)$. A similar argumentation as in Step 3 and Step 4 leads again to the estimation
    \begin{align*}
		\sup_{\mathclap{y\in\Gamma_-(x)}}\;\tau(x,y)\ge -M>-\infty\text{.}
    \end{align*}
    By Lemma \ref{lem:separatrix_transit_times} (iii), we verify that $\Gamma_2$ is a negative separatrix.
    
    \proofpart{6}{The case $x\in C_n$ with $n\in\mathcal{Q}$}
    
    If $x\in C_n$ with $n\in\mathcal{Q}$, we can apply Lemma \ref{lem:separatrix_transit_times} (i). In fact, for two arbitrarily chosen points $\eta,\zeta\in C_n$, we can apply Proposition \ref{prop:boundary_approximation} to construct a closed Jordan curve $\tilde{J}\subset\mathcal{S}$, which approximates the part of the orbit $C_n$ from $\eta$ to $\zeta$ and partially runs along $\kappa$, cf. Figure \ref{fig:geom_visual_elliptic_sector_approximation}. Hence, we get $|\tau(\eta,\zeta)|\le M$. Additionally, since $M$ is independent 
    of $n$, $\eta$ and $\zeta$, we can apply Lemma \ref{lem:transit_time_supremum} to conclude
    \begin{align*}
        \tau(C_n)=\sup_{\mathclap{\eta,\zeta\in C_n}}\;\tau(\eta,\zeta)\le M\text{.}
    \end{align*}
    Thus, $C_n$ is indeed a double-sided separatrix.
\end{proof}

Several illustrative examples can be found in \cite{broughan2003structure}. In particular, in \cite[Remark 4.4]{broughan2003structure}, the author presents an example with infinitely many separatrices, corresponding to the case $|\mathcal{Q}|=\infty$. We now present two further examples with interesting and noteworthy separatrix configurations.

\begin{example}\label{ex:F_alpha}
    We consider the polynomial vector field $F_\alpha:\C{}\to\C{}$ given by
    \begin{align*}
        F_\alpha(x):=e^{\mathrm{i}\alpha}(x-1)^2(x+1)^2
    \end{align*}
    with $\alpha\in[0,\pi)$. We have the two double equilibria $F_\alpha^{-1}(\{0\})=\{1,-1\}$, each possessing two global elliptic sectors, cf. Theorem \ref{thm:sector_number}. The term $e^{\mathrm{i}\alpha}$ rotates the direction of the vector field without changing the position of the equilibria. The values $\alpha \in [\pi,2\pi)$ have the same effect for this vector field, but with the time direction of all orbits reversed, since $e^{\mathrm{i}(\alpha+\pi)}=-e^{\mathrm{i}\alpha}$ for all $\alpha\in[0,\pi)$.
    
    One may now ask how the separatrix configuration of the two equilibria changes depending on the choice of $\alpha$. Is there a double-sided separatrix that separates $\C$ and, consequently, the respective global elliptic sectors of the two equilibria? In order to answer this question, we display the phase portrait for four different values of $\alpha$ in Figure \ref{fig:example_polynomial_alpha}.
    
    It appears to be the case that, for the case $\alpha\in[0,\pi)\setminus\{\frac{\pi}{2}\}$, there are always six separatrices ($3$ positive and $3$ negative), all of which are attached to one of the two equilibria. The time directions (positive or negative) in which the separatrices approach the equilibria alternate and are determined by the function $\lambda:\mathcal{E}(F_\alpha,2)\to\{-1,1\}$, given by $\lambda(\theta):=\cos(\argg(F_\alpha^{(2)}(a))+\theta)$ with $a\in F_\alpha^{-1}(\{0\})$, cf. \cite[Proposition 4.3]{kainz2024local}. For $\alpha\in[0,\pi)\setminus\{\frac{\pi}{2}\}$, all heteroclinic orbits are rotated depending on $\alpha$, while still connecting the two equilibria. In the case $\alpha=\frac{\pi}{2}$, however, we obtain a different separatrix configuration: The heteroclinic orbits disappear and a double-sided separatrix through the point $\mathrm{i}$ occurs. This separatrix divides the complex plane into two components, each consisting of one of the two equilibria together with its attached orbits. All these attached orbits are either homoclinic orbits within global elliptic sectors or separatrices. Moreover, the number of separatrices is reduced by $1$. However, the number of blow-ups remains $6$, as $\Gamma(\mathrm{i})$ blows up in both time directions.

    \begin{figure}[ht]
       \centering
       \includegraphics[clip]{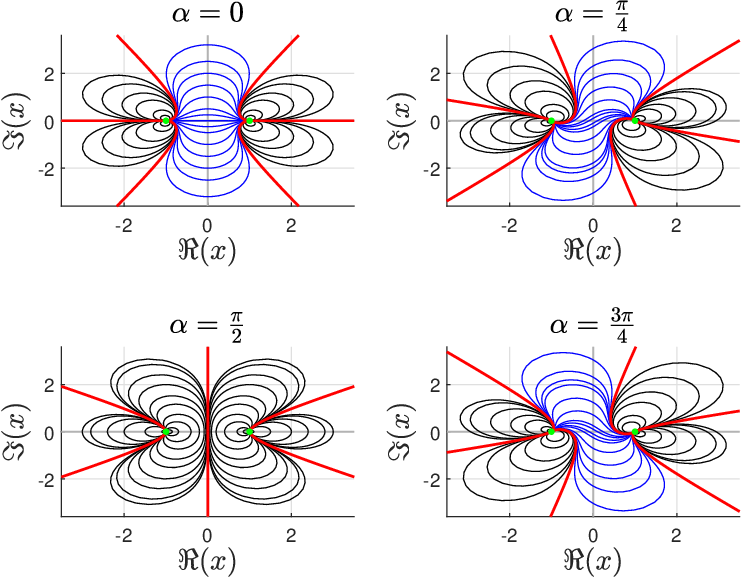}
       \caption{Local phase portrait of system \eqref{eq:planarODE} with $F_\alpha(x)=e^{\mathrm{i}\alpha}(x-1)^2(x+1)^2$ and $\alpha\in\{0,\frac{\pi}{4},\frac{\pi}{2},\frac{3\pi}{4}\}$, plotted with Matlab. The equilibria $F_\alpha^{-1}(\{0\})=\{1,-1\}$ are green. All orbits within the global elliptic sectors are black. The heteroclinic orbits are blue. The red trajectories are the separatrices on the boundary of the four elliptic sectors, cf. Theorem \ref{thm:separatrices_sectors}.}
       \label{fig:example_polynomial_alpha}
   \end{figure}

   \newpage
   
   From this analysis, we conclude that the separatrix configuration does not necessarily vary continuously under a continuous (or even holomorphic) perturbation of the vector field. It may happen that the separatrix configuration, and thus the global phase portrait, changes abruptly for specific values of $\alpha$.
\end{example}

\begin{example}
    We consider the polynomial vector field $F:\C{}\to\C{}$ given by
    \begin{equation*}
        F(x):=x^2(x-1)(x-\mathrm{i})(x-1-\mathrm{i})\text{.}
    \end{equation*}
    This vector field has already been analyzed in \cite[Example 5.8]{kainz2026geometry}. We have the equilibria $F^{-1}(\{0\})=\{0,1,\mathrm{i},1+\mathrm{i}\}$. The point $a_1=0$ is an equilibrium of order $2$, $a_2=1+\mathrm{i}$ is an attracting node, and $a_3=1$ as well as $a_4=\mathrm{i}$ are attracting foci. This leads to two global elliptic sectors in $a_1$ and three basins of attraction in $a_2$, $a_3$, and $a_4$, respectively. We illustrate the local phase portrait in Figure \ref{fig:example_polynomial}.
   
    \begin{figure}[ht]
       \centering
       \includegraphics[clip,scale=0.9]{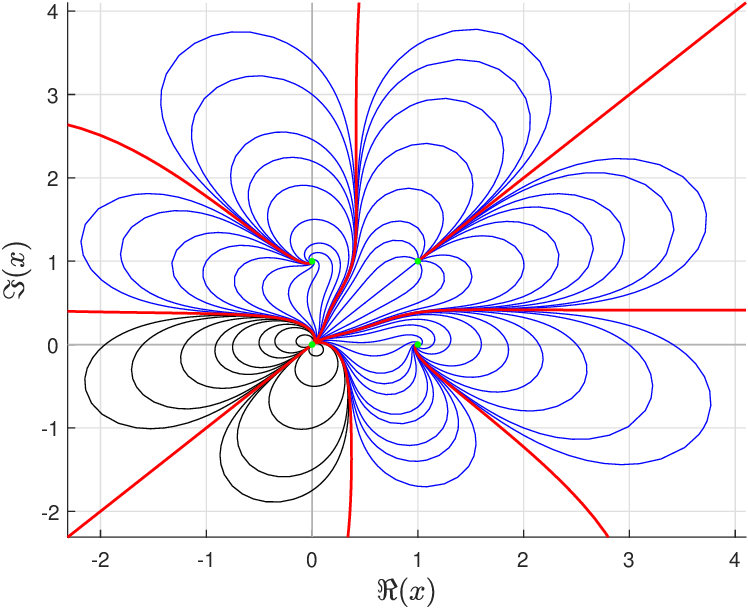}
       \caption{\cite[Fig. 5]{kainz2026geometry} Local phase portrait of system \eqref{eq:planarODE} with $F(x)=x^2(x-1)(x-\mathrm{i})(x-1-\mathrm{i})$, plotted with Matlab. The equilibria $F^{-1}(\{0\})=\{0,1,\mathrm{i},1+\mathrm{i}\}$ are green. All orbits within the basins of attraction are blue. All orbits within the global elliptic sectors are black. The red trajectories are the boundary orbits of the basins and sectors. The blue orbits within the basins are heteroclinic and connect $a_1$ and $a_j$, $j\in\{2,\ldots,4\}$.}
       \label{fig:example_polynomial}
   \end{figure}
   
    Since this example involves multiple basins of attraction as well as elliptic sectors, the corresponding separatrix configuration requires a more careful description. We denote the red boundary orbit through the point $-1-\mathrm{i}$ by $C_1$ and label the remaining red orbits $C_2,\ldots,C_8$, ordered cyclically in counterclockwise direction, cf. Figure \ref{fig:example_polynomial}.
    
    We first observe that the blue heteroclinic orbits define three \textit{heteroclinic regions} lying between $a_1$ and $a_j$ for $j\in\{2,\ldots,4\}$. These sets are specific canonical regions in the sense of \cite{markus1954global} and were introduced in \cite[Chapter 5.1]{kainz2026geometry}. They contain all heteroclinic orbits connecting one of the simple equilibria with $a_1$. The existence of the four red orbits $C_4,\ldots,C_7$ follows from \cite[Theorem 5.7]{kainz2026geometry}: indeed, we showed there that the heteroclinic regions are simply connected. As the equilibria lie on their boundaries, each region necessarily has a red boundary orbit that is not heteroclinic. A priori, and without using the theory developed in this paper, the red orbits can be identified as unbounded separatrices in the sense of Markus \cite{markus1954global} and Neumann \cite{neumann1975classification}, forming the boundary of these canonical (heteroclinic) regions. This naturally raises the question which of the red boundary orbits are also separatrices in the sense of Definition~\ref{def:separatrix}.
       
    The orbits $C_1$, $C_2$ and $C_8$ lie on the boundary of the two elliptic sectors. Hence, by Theorem \ref{thm:separatrices_sectors}, they blow up in finite time. As outlined in Example \ref{ex:F_alpha}, the time directions (positive or negative) in which the separatrices approach $a_1$ alternate. A straightforward computation yields $F(-1-\mathrm{i})=20+20\mathrm{i}$. Hence, $C_1$ is a negative and $C_2$ and $C_8$ are positive separatrices.
        
    With the methods developed in this paper, we cannot determine whether the red orbits $C_3,\ldots,C_7$ also blow up in finite time, since they do not lie on the boundary of the basins of attraction. They are located only on the boundary of the three heteroclinic regions. At this point, we conjecture that unbounded orbits on the boundary of heteroclinic regions -- and in particular the red orbits $C_3,\ldots,C_7$ -- are also separatrices in the sense of Definition \ref{def:separatrix}. A detailed investigation of this conjecture is left for future work.
\end{example}

\begin{appendices}
\section{}

\begin{proof}[Proof of Lemma \ref{lem:transit_time_supremum}]
	Fix $x\in\Gamma$. Since $\Gamma$ is not periodic, the function $\varphi_x:I(x)\to\Gamma$, $\varphi_x(t):=\Phi(t,x)$, is a bijection. The inverse function is given by $\varphi^{-1}_x(y)=\tau(x,y)$, $y\in\Gamma$. By using this, for all $y\in\Gamma$ there exists $t_y:=\tau(x,y)\in I(x)$ such that $\varphi_x(t_y)=y$. Hence we get
	\begin{align*}
		\tau(\Gamma)=\lambda(I(x))\ge\lambda\left([0,|t_y|]\right)=|\tau(x,y)|\text{.}
	\end{align*}
	Since $x$ is arbitrary, we conclude the inequality
	\begin{align}\label{eq:transit_time_supremum_inequality}
		\tau(\Gamma)\ge\sup_{\mathclap{x,y\in\Gamma}}\;\tau(x,y)\text{.}
	\end{align}
	Suppose, the inequality \eqref{eq:transit_time_supremum_inequality} is strict. First, we assume that $\tau(\Gamma)<\infty$, i.e. the maximum interval of existence of $\Gamma$ is bounded in $\R{}$. By assumption, there exists $\varepsilon>0$ such that for all $x,y\in\Gamma$ we have $\tau(\Gamma)-\varepsilon>\tau(x,y)=\varphi^{-1}_x(y)$. For fixed $z\in\Gamma$ there exist $\alpha<0$ and $\beta>0$ such that $I(z)=(\alpha,\beta)$. Choose $x:=\varphi_z\left(\alpha+\frac{\varepsilon}{2}\right)\in\Gamma$ and $x:=\varphi_z\left(\beta-\frac{\varepsilon}{2}\right)\in\Gamma$. Since the flow defines a dynamical system\footnote{cf. \cite[Chapter 3.1, Definition 1]{perko2001differential}.}, we conclude the contradiction
	\begin{align*}
		\varphi^{-1}_{x}(y)=\varphi^{-1}_{x}(z)+\varphi^{-1}_{z}(y)=-\left(\alpha+\frac{\varepsilon}{2}\right)+\beta-\frac{\varepsilon}{2}=\tau(\Gamma)-\varepsilon>\varphi^{-1}_{x}(y)\text{.}
	\end{align*}
	Thus, such a $\varepsilon$ does not exist and the inequality is not strict in the case $\tau(\Gamma)<\infty$. Assume now $\tau(\Gamma)=\infty$ and define
	\begin{align*}
		\zeta:=\sup_{\mathclap{x,y\in\Gamma}}\;|\tau(x,y)|\text{.}
	\end{align*}
	By assumption, $0\le \zeta<\infty$. Fix $x\in\Gamma$. Since $I(x)$ is unbounded and connected, there exists $t\in\{\zeta+1,-(\zeta+1)\}\cap I(x)\not=\emptyset$. But now we cleary have 
	\begin{align*}
		|\tau(x,\varphi_x(t))|=|t|=\zeta+1>\zeta\text{.}
	\end{align*}
	Hence $\varphi_x(t)\not\in\Gamma$, which is a contradiction to the fact that $\varphi_x$ is a surjection. All in all, $\zeta=\infty$ and the inequality \eqref{eq:transit_time_supremum_inequality} is not strict also in this case.
\end{proof}

\end{appendices}

\section*{Declarations}

\textbf{Competing interests:} The authors have no competing interests to declare.\\[1em]
\textbf{Funding:} We did not receive external funding for this work.

\section*{Acknowledgment}

We thank Alvaro Alvarez–Parrilla and Jes\'{u}s Muci\~{n}o-Raymundo for valuable discussions. We thank the referee for helpful comments and suggestions. We thank Francisco Fern\'{a}ndez for helpful email correspondence concerning the proof of Proposition \ref{prop:countable_orbits}.




\end{document}

%% file: homoclinicOrbitE1Lambda1.tex
\path[name path=top1, line join=round, line cap=round]
        (0.00214840746320447,0.000617783859508665) -- (0.00215673413595722,0.000620178230277909) -- (0.00221519765073649,0.000636989666851917) -- (0.00228256058239935,0.000656360169247717) -- (0.00236160515835569,0.000679089778992382) -- (0.00245657760022799,0.000706399515481797) -- (0.00257439713876841,0.00074027903345009) -- (0.00272741334116902,0.000784279504365536) -- (0.00291413139313992,0.000837971087914156) -- (0.00311363211739726,0.000895338384177338) -- (0.00332679061256701,0.000956633031552041) -- (0.0035545418863135,0.00102212389553897) -- (0.00379788495669743,0.0010920982481061) -- (0.00405788723431197,0.00116686302779119) -- (0.0043356892044187,0.00124674618607071) -- (0.00463250942962149,0.00133209812590106) -- (0.00494964989502204,0.0014232932387421) -- (0.00528850171930347,0.00152073154680525) -- (0.00565055125679315,0.0016248404577299) -- (0.00603738661727109,0.0017360766393851) -- (0.00645070463212267,0.00185492802302048) -- (0.00689231829739206,0.00198191594355347) -- (0.00736416472638492,0.00211759742638158) -- (0.00786831364670371,0.00226256763075136) -- (0.00840697647898712,0.00241746246040281) -- (0.00898251603717658,0.00258296135294189) -- (0.00959745689285906,0.00275979026017856) -- (0.0102544964491479,0.0029487248325059) -- (0.0109565167726758,0.00315059382129234) -- (0.0117065972355982,0.0033662827142161) -- (0.0125080280230597,0.00359673761949558) -- (0.0133643245653692,0.00384296941606322) -- (0.0142792429581857,0.00410605818790142) -- (0.0152567964383503,0.00438715796201065) -- (0.0163012729876256,0.0046875017708196) -- (0.0174172541415499,0.0050084070612809) -- (0.0186096350858953,0.00535128147443157) -- (0.0198836461288603,0.00571762902084296) -- (0.0212448756431548,0.00610905667915058) -- (0.0226992945785698,0.00652728144675011) -- (0.0242532826524982,0.0069741378737857) -- (0.0259136563332108,0.00745158611375483) -- (0.0276876987385221,0.00796172052642965) -- (0.0295831915808412,0.00850677887136988) -- (0.0316084492985159,0.00908915213310596) -- (0.0337723555228844,0.00971139502213866) -- (0.0360844020405731,0.0103762371992802) -- (0.038554730421362,0.0110865952746126) -- (0.0411941764933955,0.0118455856365423) -- (0.0440143178596854,0.0126565381711971) -- (0.0470275246627428,0.0135230109378835) -- (0.0502470138177928,0.0144488058727057) -- (0.0536869069493644,0.0154379856000088) -- (0.0573622922810698,0.016494891440428) -- (0.0612892907440214,0.0176241627155105) -- (0.065485126585459,0.0188307574628555) -- (0.0699682027755859,0.0201199746934575) -- (0.0747581815270293,0.0214974783458364) -- (0.0798760702573259,0.0229693231214676) -- (0.0853443133397178,0.0245419824256704) -- (0.0911868900004198,0.0262223786911235) -- (0.0974294187300445,0.0280179164327203) -- (0.104099268581177,0.0299365184796782) -- (0.111225677720537,0.0319866659636726) -- (0.118839879589059,0.034177442824176) -- (0.12697523699139,0.0365185858435358) -- (0.135667384382246,0.039020541572207) -- (0.144954378499729,0.0416945319777499) -- (0.154876857893898,0.0445526314731987) -- (0.165478201652147,0.0476078560137458) -- (0.176804866372818,0.0508743207686712) -- (0.188903523239655,0.0543664978351179) -- (0.201878582608006,0.0581159645480082) -- (0.214827667570121,0.0618633241948971) -- (0.227524479803979,0.065543976827316) -- (0.239980176673492,0.0691619349262183) -- (0.252205522819355,0.0727211266223322) -- (0.264210827447212,0.0762253746125419) -- (0.276005903354102,0.0796783817204957) -- (0.287600046014538,0.083083722567106) -- (0.299002030706632,0.0864448407396631) -- (0.310220091158749,0.0897650406420839) -- (0.321262399536212,0.0930476286465146) -- (0.332129222569288,0.0962935656329336) -- (0.342951101785458,0.0995430134176132) -- (0.354121369236862,0.102916649546812) -- (0.365650499402774,0.106421676787761) -- (0.377549149696079,0.110066082618894) -- (0.389828138339076,0.113858761057997) -- (0.402498416491408,0.117809655436421) -- (0.415571033514923,0.121929925862261) -- (0.42905709404009,0.126232145823946) -- (0.442967705229378,0.130730533247825) -- (0.457313912301174,0.13544122238181) -- (0.472106619965461,0.140382584191153) -- (0.487356496904765,0.145575604604938) -- (0.503073859776996,0.151044332061884) -- (0.519268532373243,0.156816408545161) -- (0.535949674466617,0.162923701926415) -- (0.553125573442622,0.16940306235012) -- (0.570803389870325,0.176297232184466) -- (0.58898884555778,0.183655948695389) -- (0.607685839042833,0.191537292596792) -- (0.626895968469413,0.200009356518329) -- (0.64661793474499,0.209152339496705) -- (0.666846787799539,0.21906122434399) -- (0.687572964221003,0.229849277835537) -- (0.708781043395388,0.241652755072679) -- (0.730448118591741,0.254637441218073) -- (0.752541635994549,0.269008137426687) -- (0.775016498858966,0.28502315040921) -- (0.797811188896823,0.303017931866036) -- (0.820842765279494,0.323447133429016) -- (0.84400170505746,0.346969111459278) -- (0.867156692060816,0.374653510065141) -- (0.890285394129103,0.408775954806926) .. controls (1.2,0.2) .. (in1);

%% file: homoclinicOrbitE2Lambda2.tex
\path[name path=top2, line join=round, line cap=round]
        (0.00108231479220765,0.00195599347832657)
        -- (0.00108919915792475,0.00196843512149877)
        -- (0.00111903881547431,0.0020223622931319)
        -- (0.00115347518554982,0.00208459678883244)
        -- (0.00119396509497207,0.00215777143204578)
        -- (0.00124274431761432,0.00224592678393397)
        -- (0.0013034834763901,0.00235569651016509)
        -- (0.00138280511494699,0.00249904908079551)
        -- (0.00147747161577868,0.00267013337120274)
        -- (0.00157861896216431,0.00285293005039017)
        -- (0.00168669083120996,0.00304824094563466)
        -- (0.00180216127404838,0.00325692277711218)
        -- (0.00192553679523747,0.00347989091585197)
        -- (0.002057358574514,0.00371812339895929)
        -- (0.00219820484064804,0.00397266521971918)
        -- (0.00234869340781091,0.00424463291139944)
        -- (0.00250948438558246,0.00453521944485937)
        -- (0.00268128307448477,0.00484569946144724)
        -- (0.00286484305974362,0.00517743486414002)
        -- (0.00306096951684836,0.00553188079145041)
        -- (0.00327052274340984,0.00591059200030499)
        -- (0.00349442193280897,0.00631522968589113)
        -- (0.00373364920618895,0.00674756876838679)
        -- (0.00398925392047756,0.00720950567853512)
        -- (0.00426235727133677,0.00770306667621337)
        -- (0.00455415721123088,0.00823041673848329)
        -- (0.00486593370418672,0.00879386905610709)
        -- (0.00519905434029661,0.00939589518018136)
        -- (0.00555498033459354,0.0100391358633913)
        -- (0.00593527293661398,0.0107264126434324)
        -- (0.00634160027876682,0.0114607402194)
        -- (0.00677574469355213,0.0122453396754207)
        -- (0.00723961053173215,0.0130836526095101)
        -- (0.0077352325157563,0.0139793562296061)
        -- (0.00826478466509299,0.0149363794829566)
        -- (0.00883058983263407,0.0159589202895592)
        -- (0.00943512989402413,0.017051463955172)
        -- (0.0100810566346395,0.0182188028445595)
        -- (0.0107712033820136,0.0194660574011243)
        -- (0.0115085974347914,0.0207986986049245)
        -- (0.0122964733428139,0.0222225719673057)
        -- (0.0131382870966958,0.0237439231670066)
        -- (0.0140377312892983,0.0253694254396406)
        -- (0.01499875131582,0.0271062088399323)
        -- (0.0160255626838731,0.0289618915040008)
        -- (0.0171226695099038,0.03094461304733)
        -- (0.0182948842836856,0.0330630702428536)
        -- (0.0195473489884134,0.0353265551327576)
        -- (0.0208855576702034,0.0377449957371406)
        -- (0.0223153805576151,0.0403289995324675)
        -- (0.0238430898392562,0.0430898998826761)
        -- (0.0254753872156921,0.0460398056156506)
        -- (0.0272194333509122,0.0491916539472423)
        -- (0.0290828793586725,0.0525592669636704)
        -- (0.0310739004703829,0.0561574118803293)
        -- (0.0332012320441575,0.0600018652998567)
        -- (0.0354742080896285,0.0641094816934973)
        -- (0.0379028025007388,0.0684982663255022)
        -- (0.0404976732097848,0.0731874528279855)
        -- (0.0432702095015987,0.0781975856097229)
        -- (0.0462325827584876,0.0835506072417962)
        -- (0.0493978009465308,0.0892699508987588)
        -- (0.0527797672050411,0.0953806378363318)
        -- (0.0563933429675603,0.101909379741971)
        -- (0.0602544161304403,0.108884685584096)
        -- (0.0643799749018687,0.116336972283332)
        -- (0.068788188121351,0.124298678098807)
        -- (0.0734984930528051,0.132804377015188)
        -- (0.0785316919455002,0.141890891564542)
        -- (0.0839100590568329,0.151597400330391)
        -- (0.0896574603826284,0.161965534737969)
        -- (0.0957994891051889,0.173039457474173)
        -- (0.102363620832768,0.184865911795783)
        -- (0.109379394188503,0.197494226815068)
        -- (0.116878624385627,0.210976258286116)
        -- (0.124895660346773,0.225366237106329)
        -- (0.133467700043501,0.240720488361608)
        -- (0.142635184558575,0.257096972063285)
        -- (0.15244229964383,0.274554582844532)
        -- (0.162937625321784,0.29315213054393)
        -- (0.174174990885398,0.312946908697823)
        -- (0.186214616674054,0.333992747208355)
        -- (0.199124658365198,0.356337445058986)
        -- (0.212983318693427,0.380019497810934)
        -- (0.22788176208746,0.405064083097597)
        -- (0.243928170169772,0.431478353237274)
        -- (0.261253429240894,0.459246204601366)
        -- (0.280019183816883,0.488322823911791)
        -- (0.300429412699735,0.518629397102694)
        -- (0.322747504237006,0.550048326664377)
        -- (0.347322582644599,0.582419070383857)
        -- (0.374633069352616,0.615534293202009)
        -- (0.405367306474341,0.649136265733626)
        -- (0.440589009205476,0.682907556859906)
        -- (0.4823962420081,0.716627632063766)
        -- (0.530612668121931,0.747386591257611) .. controls (0.35,1.2) .. (in2);